\documentclass[a4paper,reqno]{amsart}
\usepackage{amsmath}
\usepackage{amssymb}
\usepackage{amsthm}
\usepackage{latexsym}
\usepackage{amscd}
\usepackage{xypic}
\xyoption{curve}
\usepackage{ifthen} 
\usepackage{hyperref} 
\usepackage{graphicx}
\usepackage{enumerate}
\usepackage{enumitem}
\usepackage{rotating}
\usepackage{xcolor}
\usepackage{mathtools}
\usepackage{todonotes}

\usepackage{float}
\restylefloat{table}

\numberwithin{equation}{section}

\def\cocoa{{\hbox{\rm C\kern-.13em o\kern-.07em C\kern-.13em o\kern-.15em A}}}

\newtheorem{theorem}{Theorem}[section]
\newtheorem{question}[theorem]{Question}
\newtheorem{lemma}[theorem]{Lemma}
\newtheorem{proposition}[theorem]{Proposition}
\newtheorem{corollary}[theorem]{Corollary}

\theoremstyle{definition}
\newtheorem{remark}[theorem]{Remark}
\newtheorem{definition}[theorem]{Definition}

\newtheorem{construction}[theorem]{Construction}

\newcommand {\sHom}{\mathcal{H}\kern -0.25ex{\mathit om}}
\newcommand {\sExt}{\mathcal{E}\kern -0.25ex{\mathit xt}}
\newcommand {\sTor}{\mathcal{T}\kern -0.25ex{\mathit or}}
\newcommand {\im}{\mathrm{im}}
\newcommand {\rk}{\mathrm{rk}}

\newcommand {\Ext}{\mathrm{Ext}}
\newcommand {\Hom}{\mathrm{Hom}}
\newcommand {\Hilb}{\mathcal{H}\kern -0.25ex{\mathit ilb\/}}
\newcommand {\Mov}{\overline{\mathrm{Mov}}}

\newcommand {\cK}{\mathcal{K}}
\newcommand {\cA}{\mathcal{A}}

\newcommand {\cU}{\mathcal{U}}
\newcommand {\cV}{\mathcal{V}}
\newcommand {\cW}{\mathcal{W}}
\newcommand {\cP}{\mathcal{P}}
\newcommand {\cQ}{\mathcal{Q}}
\newcommand{\cC}{{\mathcal C}}

\newcommand{\cE}{{\mathcal E}}
\newcommand{\cF}{{\mathcal F}}
\newcommand{\cM}{{\mathcal M}}
\newcommand{\cN}{{\mathcal N}}
\newcommand{\cO}{{\mathcal O}}
\newcommand{\cG}{{\mathcal G}}

\newcommand{\cI}{{\mathcal I}}

\newcommand {\bZ}{\mathbb{Z}}

\newcommand {\bC}{\mathbb{C}}
\newcommand {\bP}{\mathbb{P}}

\newcommand{\Pic}{\operatorname{Pic}}

\def\p#1{{\bP^{#1}}}

\def\mapright#1{\mathbin{\smash{\mathop{\longrightarrow}
\limits^{#1}}}}

\title[Instanton bundles]{Instanton bundles\\ on two Fano threefolds of index $1$}

\thanks{The first author is a member of GNSAGA group of INdAM, of PRIN 2015 \lq Geometry of Algebraic Varieties\rq, cofinanced by MIUR and is supported by the framework of the MIUR grant Dipartimenti di Eccellenza 2018-2022 (E11G18000350001). The second author is supported by Narodowe Centrum Nauki 2018/30/E/ST1/00530.}



\author[Gianfranco Casnati, Ozhan Genc]{Gianfranco Casnati, Ozhan Genc}

\begin{document}

\begin{abstract}
We deal with instanton bundles on the product $\p1\times\p2$ and the blow up of $\p3$ along a line. We give an explicit construction leading to instanton bundles. Moreover, we also show that they correspond to smooth points of a unique irreducible component of their moduli space.
\end{abstract}

\maketitle

\section{Introduction}
A smooth irreducible closed subscheme $X\subseteq \p N$ of dimension $3$ is called a {\sl Fano threefold} if its anticanonical line bundle $\omega_X^{-1}$ is ample (see \cite{I--P} for the results about Fano threefold mentioned in what follows). The {\sl index} $i_X$ of a Fano threefold is the greatest integer such that $\omega_X\cong\cO_X(-i_Xh)$ for some ample line bundle $\cO_X(h)\in\Pic(X)$. Such a line bundle $\cO_X(h)$ is uniquely determined and it is called the {\sl fundamental line bundle of $X$}. 

One has  $1\le i_X\le 4$ and for each $i_X$ in this range there is a finite number of deformation families of Fano threefolds of index $i_X$. E.g., if $i_X=4,3$ if and only if $X$ is isomorphic to either $\p3$, or the smooth quadric in $\p4$, respectively. There exist $8$ deformation families of Fano threefolds with $i_X=2$ and $95$ with $i_X=1$.

In the seminal paper \cite{A--D--H--M} the authors introduced for the first time instanton bundles on $\p3$ as rank $2$ bundles $\cE$ such that $c_1(\cE)=0$ and $h^0\big(\p3,\cE\big)=h^1\big(\p3,\cE(-2)\big)=0$. Since then, instanton bundles have been widely studied, especially from the viewpoint of the smoothness and connectedness of their moduli space. 

Also a number of generalizations of instantons appeared. E.g. in \cite{Fa} (see also \cite{Kuz}) the author extends the notion of instanton bundle to each Fano threefold with cyclic Picard group as those rank two bundles such that $c_1(\cE)=(2q_X-i_X)h$ and $h^0\big(X,\cE\big)=h^1\big(X,\cE(-q_Xh)\big)=0$, where
$$
q_X:=\left[\frac {i_X}2\right].
$$
The author also studied therein instanton bundles on several Fano threefolds $X$ with indices $1\le i_X\le3$.
In \cite{M--M--PL, C--C--G--M} the authors extended the definition of instanton bundle to each Fano threefold. 

In order to understand such a definition we recall the notion of $\mu$--(semi)stability. For each sheaf $\cF$ on $X$ the {\sl slope} of $\cF$ with respect to $\cO_X(h)$ is the rational number $\mu(\cF):=c_1(\cF)h^2/\rk(\cF)$. We say that the coherent torsion--free sheaf $\cF$ is {\sl $\mu$--stable} (resp. {\sl $\mu$--semistable}) with respect to $\cO_X(h)$ if  $\mu(\mathcal G) < \mu(\cF)$ (resp. $\mu(\mathcal G) \le \mu(\cF)$) for each subsheaf $\mathcal G$ with $0<\rk(\mathcal G)<\rk(\cF)$.

\begin{definition}
\label{dInstanton}
Let $X$ be a Fano threefold.

A vector bundle $\cE$ of rank $2$ on $X$ is called an instanton bundle if the following properties hold:
\begin{itemize}
\item $c_1(\cE)=(2q_X-i_X)h$;
\item $\cE$ is $\mu$--semistable with respect to $\cO_{X}(h)$ and $h^0\big(X,\cE\big)=0$;
\item $h^1\big(X,\cE(-q_Xh)\big)=0$;
\end{itemize}
The class $c_2(\cE)\in A^2(X)$ is called the {charge} of $\cE$.
\end{definition}

When $\Pic(X)\cong\bZ$ and $\cE$ is a rank $2$ bundle with $c_1(\cE)\in\{\ 0,-h\ \}$, then the vanishing $h^0\big(X,\cE\big)=0$ is equivalent to the $\mu$--stability of $\cE$. This is no longer true if  $\rk\Pic(X)\ge2$.

Nevertheless, a bundle $\cE$ which is either $\mu$--stable with $c_1(\cE)=0$, or $\mu$--semistable with $c_1(\cE)=-h$ always satisfies $h^0\big(X,\cE\big)=0$. In particular, the latter vanishing on Fano threefolds with odd $i_X$ is an immediate consequence of the other properties in Definition \ref{dInstanton}.

In \cite{M--M--PL} the authors studied bundles which are instanton in the sense of the previous definition on the {\sl flag threefold}, i.e. the general hyperplane section of the Segre image of $\p2\times\p2$. 
In \cite{C--C--G--M} a similar description has been given for the blow up of $\p3$ at a point, where the condition on a class $\zeta\in A^2(X)$ for being the charge of an instanton bundle are also given. An analogous study on $\p1\times\p1\times\p1$ is the object of \cite{A--M}. 

All these threefolds are important examples of Fano threefolds of index $2$ and they complete the analysis of instanton bundles on Fano threefolds of index $2$ with very ample fundamental divisor.

In the paper  \cite{C--C--G--M} the authors introduced the following definitions, where $\Lambda$ denotes the Hilbert scheme of lines in $X$.

\begin{definition}
\label{dEarnest}
Let $\cE$ be an instanton bundle on a Fano threefold $X$.
\begin{itemize}
\item We say that $\cE$ is generically trivial  on $\Lambda$ (resp. on the component $\Lambda_0\subseteq\Lambda$) if $h^1\big(L,\cE((i_X-2q_X-1)h)\otimes\cO_L\big)=0$ when $L\in\Lambda$ (resp. $L\in\Lambda_0$) is general.
\item We say that $\cE$ is  earnest if $h^1\big(X,\cE(-q_Xh-D)\big)=0$ when $\vert D\vert\ne\emptyset$ contains smooth integral elements.
\end{itemize}
\end{definition}

If $i_X$ is even, generically trivial instanton bundles on the component $\Lambda_0\subseteq\Lambda$ are the instanton bundles such that $\cE\otimes\cO_L\cong\cO_{\p1}^{\oplus2}$  for each general $L\in\Lambda_0$, while when $i_X$ is odd, the ones such that $\cE\otimes\cO_L\cong\cO_{\p1}(-1)\oplus\cO_{\p1}$. Each instanton bundle is generically trivial if $i_X\ge3$ (see \cite{Fa}). When $i_X\le2$ the generic triviality of each instanton bundle has been conjectured in \cite[Section 3.7 and Conjecture 3.16]{Kuz}. 

The notion of earnest instanton bundle is related to the $\mu$--semistability of its restriction to general hypersurface sections: see the introduction of \cite{C--C--G--M} for some details. In particular, if $\Pic(X)\cong\bZ$ each instanton bundle is earnest, thanks to a theorem of Maruyama (see \cite[Examples 3.2 and 3.3]{C--C--G--M}). One can prove that the same is true when $X$ is either the flag threefold (see \cite[Example 3.4]{C--C--G--M}), or $\p1\times\p1\times\p1$ (see \cite{A--M}).

When $X$ is the blow up of $\p3$ at a point, it is not immediate whether instanton bundles are earnest or not. Indeed, in \cite{C--C--G--M} the authors are able only to prove that the apparently infinite set of vanishing in the above definition reduces to the single vanishing for the exceptional divisor of the blow up. Moreover, the existence of earnest instanton bundles on that Fano threefold is proved for every admissible choice of the charge. 

In the present paper we focus our attention on $F_0:=\p1\times\p2$ and on the blow up $F_1$ of $\p3$ along a line $R$. Notice that $F_e$ is a Fano threefold with $i_{F_e}=1$ for $e=0,1$. We have a natural isomorphism $F_e\cong \bP(\cP_e)\mapright\pi\p1$, where $\cP_e:=\cO_{\p1}^{\oplus2}\oplus\cO_{\p1}(e)$: throughout the whole paper, following \cite{Ha2},  for each coherent sheaf $\cG$ on $F_e$ we set  $\cP(\cG):=\mathrm{Proj}(\mathrm{Sym}(\cG))$. We denote by $\xi_e$ and $f$ the classes of $\cO_{\bP(\cP_e)}(1)$ and $\pi^*\cO_{\p1}(1)$ respectively. Thus we have an isomorphism
$$
A(F_e)\cong\bZ[\xi_e,f](f^2,\xi_e^3-e\xi_e^2 f).
$$
The fundamental line bundle is $\cO_{F_e}(3\xi_e+(2-e)f)$. If $e=1$ it corresponds to the quartic surfaces throughout $R$. From now on $E\subseteq F_1$ denotes the exceptional divisor of the blow up. The arguments used in the two cases $e=0$ and $e=1$ are definitely similar. Indeed the two threefolds behave in a very similar way, as we show in Section \ref{sFano}. 

We first deal with the threefold $F_1$ in Sections \ref{sMonad}, \ref{sInstanton}, \ref{sEarnest} and \ref{sFinal}. We then describe the changes in the arguments which are necessary for dealing with $F_0$ in the last Section \ref{sSegre}. 

Section \ref{sGeneral} contains some general and well--known results concerning instanton bundles on Fano threefolds, while Section \ref{sFano} is devoted to list some results on the threefolds $F_e$.

In Section \ref{sMonad} we first prove that the coefficients $\alpha$ and $\beta$ for the charge $\alpha\xi_1^2+\beta\xi_1 f$ of an instanton bundle satisfy a list of restrictions: among them $\alpha\ge2$, $\alpha+\beta\ge4$ and, for earnest instanton bundles, $\beta\ge1$. Then we prove the existence of a monad associated to each instanton bundle on $F_1$. 

More precisely, for every choice of integers $\alpha,\beta,\gamma,\delta$ with $\alpha\ge2$, $\gamma\ge0$ and
\begin{equation}
\label{abcd}
\beta\ge\max\left\{\ 4-\alpha,\ \alpha-\delta-2,\ 1-\gamma\ \right\},\qquad \delta\ge2\gamma,
\end{equation}
we set
\begin{gather*}
\cC^{-1}_1:=\cO_{F_1}(-2\xi_1-f)^{\oplus\alpha+\beta-4}\oplus\cO_{F_1}(-2\xi_1)^{\oplus\gamma},\\
\cC^0_1:=\cO_{F_1}(-2\xi_1)^{\oplus\beta+\gamma-1}\oplus\Omega_{F_1\vert\p1}(-f)^{\oplus\alpha-2}\oplus\cO_{F_1}(-\xi_1)^{\oplus\delta+\beta-\alpha+2},\\
\cC^1_1:=\cO_{F_1}(-\xi_1-f)^{\oplus\gamma}\oplus\cO_{F_1}(-\xi_1)^{\oplus\delta-2\gamma}\oplus\cO_{F_1}(-\xi_1+f)^{\oplus\beta+\gamma-1}.
\end{gather*}

Our first main result is as follows.

\begin{theorem}
\label{tSimplify}
Let $\cE$ be an instanton bundle with charge $\alpha\xi_1^2+\beta  \xi_1 f$ on $F_1$.

Then $\cE$ is the cohomology of a monad $\cC^\bullet_1$ of the form
\begin{equation}
\label{Monad}
0\longrightarrow \cC^{-1}_1\longrightarrow \cC^0_1\longrightarrow\cC^1_1\longrightarrow0
\end{equation}
where $\gamma:=h^1\big({F_1},\cE(-\xi_1+f)\big)$, $\delta:=h^1\big({F_1},\cE(-\xi_1+2f)\big)$.

Conversely, if the cohomology $\cE$ of the monad $\cC^\bullet_1$ is a $\mu$--semistable bundle for some integers $\alpha,\beta,\gamma,\delta$, then $\cE$ is an instanton bundle with charge $\alpha\xi_1^2+\beta  \xi_1 f$ on $F_1$ such that
\begin{enumerate}
\item $h^1\big(F_1,\cE(-\xi_1+f)\big)=\gamma$;
\item $h^1\big(F_1,\cE(-\xi_1+2f)\big)\le \delta$;
\item $h^1\big(F_1,\cE(-D)\big)=0$ for each integral smooth effective divisor $D\not\in\vert\xi-f\vert$.
\end{enumerate}
\end{theorem}

As an almost immediate by--product of the above monadic description we characterize earnest instanton bundles $\cE$ as the ones such that the single vanishing
\begin{equation*}
h^1\big(F_1,\cE(-\xi+f)\big)=0
\end{equation*}
holds (see Corollary \ref{cSimplify}). Moreover, we also prove that the charge $\alpha\xi_1^2+\beta\xi_1 f$ of an instanton bundle always satisfies $4\alpha+3\beta\ge15$.

In Section \ref{sInstanton} we deal with the existence of instanton bundles for all the admissible values of their charge $\alpha\xi_1^2+\beta\xi_1 f$, i.e. $\alpha\ge2$, $\alpha+\beta\ge4$ and $4\alpha+3\beta\ge15$. More precisely, we describe therein a construction (see Construction \ref{conInstanton}) leading to certain bundles $\cE$ of rank $2$ with $c_2(\cE)=\alpha\xi_1^2+\beta\xi_1 f$ and then we prove the following result.

\begin{theorem}
\label{tInstanton}
If $\alpha\ge2$, $\alpha+\beta\ge4$, $4\alpha+3\beta\ge15$, then the bundle $\cE$ obtained via Construction \ref{conInstanton} is a generically trivial $\mu$--stable instanton bundle $\cE$ with charge $\alpha\xi_1^2+\beta  \xi_1 f$ on $F_1$ such that 
$$
\dim\Ext^1_{F_1}\big(\cE,\cE\big)=8\alpha+6\beta-30,\qquad \Ext^2_{F_
1}\big(\cE,\cE\big)=\Ext^3_{F_1}\big(\cE,\cE\big)=0.
$$
\end{theorem}
We then conclude the section proving that all the bundles above  represent points in a single component of the moduli space of instanton bundles.

It is noteworthy to remark that Construction \ref{conInstanton} often returns non earnest bundles. E.g. such bundles are certainly non earnest when either $\beta\le 0$ (indeed, in this case, $h^1\big(F_1,\cE(-\xi+f)\big)=\gamma\ge1$ thanks to Corollary \eqref{cBound} or  Inequalities \eqref{abcd} above), or $\alpha\ge4$ without restrictions on $\beta$ (see Remark \ref{rNonEarnest}). Thus it is quite natural to ask if it is possible to find different constructions leading to earnest instanton bundle. 

A first trivial remark is that this is certainly not possible when $\beta\le0$, because $1-\gamma\le\beta$. But even if $\beta\ge1$, we are not able of deducing the existence of earnest instanton bundles, because of the aforementioned Remark \ref{rNonEarnest}.

For this reason, in Section \ref{sEarnest}, we describe a second alternative construction (see Construction \ref{conEarnest}) which returns earnest instanton bundles for all the admissible values of $\alpha$ and $\beta$, i.e. when $\alpha\ge2$, $\beta\ge1$ and $4\alpha+3\beta\ge15$. More precisely, we prove the existence of bundles $\cE$ of rank $2$ with $c_2(\cE)=\alpha\xi_1^2+\beta\xi_1 f$ such that the following result holds true.

\begin{theorem}
\label{tEarnest}
If $\alpha\ge2$, $\beta\ge1$, $4\alpha+3\beta\ge15$, then the bundle $\cE$ obtained via Construction \ref{conEarnest} is an earnest, generically trivial, $\mu$--stable instanton bundle $\cE$ with charge $\alpha\xi_1^2+\beta  \xi_1 f$ on $F_1$ such that 
$$
\dim\Ext^1_{F_1}\big(\cE,\cE\big)=8\alpha+6\beta-30,\qquad \Ext^2_{F_1}\big(\cE,\cE\big)=\Ext^3_{F_1}\big(\cE,\cE\big)=0.
$$
\end{theorem}
As in the previous case we finally prove that all the bundles above represent points in a single component of the moduli space of instanton bundles.

In what follows we will denote by $\cI_{F_1}(\alpha\xi^2+\beta  \xi f)$ the locus of points representing instanton bundles with charge $\alpha\xi^2+\beta  \xi f$ in the moduli space $\cM_{F_1}(2;0,\alpha\xi^2+\beta  \xi f)$ of vector bundles $\cE$ of rank $2$ with $c_1(\cE)=0$ and $c_2(\cE)=\alpha\xi^2+\beta  \xi f$ which are $\mu$--stable with respect to $\cO_{F_1}(h)$.

In view of the irreducibility of the moduli space of instanton bundles on $\p3$ recently proved in \cite{Tik1,Tik2}, and the results listed above it is natural to ask whether $\cI_{F_1}(\alpha\xi^2+\beta  \xi f)$ is irreducible as well, or at least if Constructions \ref{conInstanton} and \ref{conEarnest} actually give bundles in the same component when $\alpha\ge2$, $\beta\ge1$, $4\alpha+3\beta\ge15$.

We are not able to answer the above natural questions. Nevertheless, in Section \ref{sFinal} we deal with them, giving very partial answers in few particular cases.

In Section \ref{sSegre}, we turn our attention to the threefold $F_0$. The following two theorems are proved with the same arguments used in Theorem \ref{tSimplify}, in Construction \ref{conEarnest} and in Theorem \ref{tEarnest}.

More precisely, for every choice of integers $\alpha,\beta$ with $\alpha\ge2$, $\beta\ge3$, $\alpha+\beta\ge6$ we set
\begin{gather*}
\cC^{-1}_0:=\cO_{F_0}(-2\xi_0-f)^{\oplus\alpha+\beta-6},\\
\cC^0_0:=\cO_{F_0}(-2\xi_0)^{\oplus\beta-3}\oplus\Omega_{F_0\vert \p1}(-f)^{\oplus\alpha-2}\oplus\cO_{F_0}(-\xi_0-f)^{\oplus\gamma},\\
\cC^1_0:=\cO_{F_0}(-\xi_0-f)^{\oplus\alpha-\beta+\gamma}\oplus\cO_{F_0}(-\xi_0)^{\oplus\beta-3}.
\end{gather*}

The first main result of Section \ref{sSegre} is as follows.

\begin{theorem}
\label{tSimplifySegre}
Let $\cE$ be an instanton bundle with charge $\alpha\xi_0^2+\beta  \xi_0 f$ on $F_0$.

Then $\cE$ is the cohomology of a monad $\cC^\bullet_0$ of the form
\begin{equation}
\label{MonadSegre}
0\longrightarrow \cC^{-1}_0\longrightarrow \cC^0_0\longrightarrow\cC^1_0\longrightarrow0,
\end{equation}
where $\gamma:=h^1\big({F_0},\cE(\xi_0-f)\big)$.

Conversely, if the cohomology $\cE$ of the monad $\cC^\bullet_0$ is a $\mu$--semistable bundle for some integers $\alpha,\beta,\gamma$, then $\cE$ is an earnest instanton bundle with charge $\alpha\xi_0^2+\beta  \xi_0 f$ such that $h^1\big({F_0},\cE(\xi_0-f)\big)\le \gamma$.
\end{theorem}

Then we describe a construction (see Construction \ref{conSegre}) leading to bundles $\cE$ of rank $2$ with $c_2(\cE)=\alpha\xi_0^2+\beta\xi_0 f$ such that the following result holds.

\begin{theorem}
\label{tSegre}
If $\alpha\ge2$, $\beta\ge3$, $\alpha+\beta\ge6$, then the bundle $\cE$ obtained via Construction \ref{conSegre} is an earnest, generically trivial, $\mu$--stable instanton bundle $\cE$ with charge $\alpha\xi_0^2+\beta  \xi_0 f$ on $F_0$ such that 
$$
\dim\Ext^1_{F_0}\big(\cE,\cE\big)=4\alpha+6\beta-30,\qquad \Ext^2_{F_0}\big(\cE,\cE\big)=\Ext^3_{F_0}\big(\cE,\cE\big)=0.
$$
\end{theorem}

\subsection{Acknowledgements}
The authors would like to express their thanks to the referee for her/his criticisms, questions, remarks and suggestions which have considerably improved the whole exposition.

\section{General facts}
\label{sGeneral}
We list below some general helpful results used throughout the whole paper.  Let $X$ be any smooth projective variety with canonical line bundle $\omega_X$.

If $\cG$ and $\mathcal H$ are coherent sheaves on $X$, then  the Serre duality holds
\begin{equation}
\label{Serre}
\Ext_X^i\big(\mathcal H,\cG\otimes\omega_X\big)\cong \Ext_X^{\dim(X)-i}\big(\cG,\mathcal H\big)^\vee
\end{equation}
(see \cite[Proposition 7.4]{Ha3}). 

Let $\cF$ be a vector bundle of rank $2$ on $X$ and let $s\in H^0\big(X,\cF\big)$. In general its zero--locus
$(s)_0\subseteq X$ is either empty or its codimension is at most
$2$. We can always write $(s)_0=S\cup Z$
where $Z$ has codimension $2$ (or it is empty) and $S$ has pure codimension
$1$ (or it is empty). In particular $\cF(-S)$ has a section vanishing
on $Z$, thus we can consider its Koszul complex 
\begin{equation}
  \label{seqSerre}
  0\longrightarrow \cO_X(S)\longrightarrow \cF\longrightarrow \cI_{Z\vert X}(-S)\otimes\det(\cF)\longrightarrow 0.
\end{equation}
Sequence \ref{seqSerre} tensored by $\cO_Z$ yields $\cI_{Z\vert X}/\cI^2_{Z\vert X}\cong\cF^\vee(S)\otimes\cO_Z$, whence
\begin{equation}
\label{Normal}
\cN_{Z\vert X}\cong\cF(-S)\otimes\cO_Z.
\end{equation}
If $S=0$, then $Z$ is locally complete intersection inside $X$, because $\rk(\cF)=2$. In particular, it has no embedded components.

The above construction can be reversed by the Serre correspondence as follows.

\begin{theorem}
  \label{tSerre}
  Let $Z\subseteq X$ be a local complete intersection subscheme of codimension $2$.
  
  If $\det(\cN_{Z\vert X})\cong\cO_Z\otimes\mathcal L$ for some $\mathcal L\in\Pic(X)$ such that $h^2\big(X,\mathcal L^\vee\big)=0$, then there exists a vector bundle $\cF$ of rank $2$ on $X$ such that:
  \begin{enumerate}
  \item $\det(\cF)\cong\mathcal L$;
  \item $\cF$ has a section $s$ such that $Z$ coincides with the zero locus $(s)_0$ of $s$.
  \end{enumerate}
  Moreover, if $H^1\big(X,{\mathcal L}^\vee\big)= 0$, the above two conditions  determine $\cF$ up to isomorphism.
\end{theorem}
\begin{proof}
See \cite{Ar}.
\end{proof}

The Riemann--Roch formula for a vector bundle $\cF$ on a threefold $X$ is
\begin{equation}
  \label{RRgeneral}
  \begin{aligned}
    \chi(\cF)&=\rk(\cF)\chi(\cO_X)+{\frac16}(c_1(\cF)^3-3c_1(\cF)c_2(\cF)+3c_3(\cF))\\
    &-{\frac14}(\omega_Xc_1(\cF)^2-2\omega_Xc_2(\cF))+{\frac1{12}}(\omega_X^2c_1(\cF)+c_2(\Omega_{X}) c_1(\cF))
  \end{aligned}
\end{equation}
(see \cite[Theorem A.4.1]{Ha2}). 

We close the section by listing some results on instanton bundles which hold true on Fano threefold $X$ with $i_X=1$. The first result is the following trivial specialization of Formula \eqref{Serre} for bundles $\cF$ with $c_1(\cF)=-h$:
\begin{equation}
\label{Serre1}
h^i\big(X,\cF(D)\big)=h^{3-i}\big(X,\cF(-D)\big)
\end{equation}
for each line bundle $\cO_X(D)\in\Pic(X)$. In particular $h^0\big(X,\cF\big)=h^3\big(X,\cF\big)$, hence the following lemma is easy to prove. Moreover, $\chi(\cO_X)=1$ and
\begin{equation}
\label{c_1c_2}
c_2(\Omega_{X})c_1(\cF)={-24}
\end{equation}
(see \cite[Exercise A.6.7]{Ha2}). 

\begin{lemma}
\label{lNatural}
Let $X$ be a Fano threefold with $i_X=1$.

A $\mu$--semistable bundle $\cE$ of rank $2$ on $X$ such that $c_1(\cE)=-h$ is an instanton bundle if and only if $h^i\big(X,\cE\big)=0$ for each $i$.
\end{lemma}
\begin{proof}
If $\cE$ is an instanton bundle the statement follows from the definition and Equality \eqref{Serre1}. The converse is true by definition.
\end{proof}

If $\cE$ is an instanton bundle on $X$, then we know that $\cE\otimes\cO_H$ is $\mu$--semistable for each a general hyperplane section $H$ of $X$ thanks to \cite[Theorem 3.1]{Ma}), hence the Bogomolov inequality for $\cE\otimes\cO_H$ yields 
\begin{equation}
\label{Minimal}
c_2(\cE) h\ge\frac{\deg(X)}4.
\end{equation}
Moreover if $\cE$ is also simple, then $\dim\Hom_{X}\big(\cE,\cE\big)=1$. It follows from Equality \eqref{Serre} that
$$
\Ext^3_{X}\big(\cE,\cE\big)^\vee\cong\Hom_{X}\big(\cE,\cE(-h)\big)\subseteq\Hom_{X}\big(\cE,\cE\big).
$$
If $\varphi\in\Hom_{X}\big(\cE,\cE(-h)\big)$, then $\det(\varphi)\in H^0\big(X,\cO_X(-2h)\big)=0$. Since, being $\cE$ simple, each non zero endomorphism of $\cE$ is an automorphism, it follows that $\varphi=0$, i.e.
\begin{equation}
\label{Ext3}
\dim\Ext^3_{X}\big(\cE,\cE\big)=0
\end{equation}
Thus Formula \eqref{RRgeneral} for $\cE\otimes\cE^\vee$ yields
\begin{equation}
\label{Ext12}
\dim\Ext^1_{X}\big(\cE,\cE\big)-\dim\Ext^2_{X}\big(\cE,\cE\big)=2c_2(\cE) h-\frac{\deg(X)}2-3.
\end{equation}

\section{The threefolds $F_0$ and $F_1$}
\label{sFano}
In this section we list all the basic results on the two threefolds $F_0$ and $F_1$ that we will use in the next sections.

The threefold $F_0=\p1\times\p2$ is trivially endowed with the projections $\sigma_0\colon F_0\to \p2$ and $\pi\colon F_0\cong\bP(\cP_0)\to\p1$, where $\cP_0:=\cO_{\p1}^{\oplus3}$. The classes $\xi_0$ and $f$ of $\sigma^*_0\cO_{\p2}(1)\cong \cO_{\bP(\cP_0)}(1)$ and $\pi^*\cO_{\p1}(1)$ are obviously globally generated. 

Also $F_1$ is endowed with two natural morphisms, the blow up map $\sigma_1\colon F_1\to \p3$ and the natural projection $\pi\colon F_1\cong\bP(\cP_1)\to\p1$, where $\cP_1:=\cO_{\p1}^{\oplus2}\oplus\cO_{\p1}(1)$. Since the normal bundle of the blown up $R$ inside $\p3$ satisfies $\cN_{R\vert\p3}\cong\cO_{\p1}(1)^{\oplus2}$, it follows that $E:=\sigma_{1}^{-1}(R)\cong\p1\times\p1$ and $\sigma_1$ induces an isomorphism $F_1\setminus \sigma_1^{-1}(R)\cong\p3\setminus\{\ R\ \}$. Recall that $\xi_1$ and $f$ are the classes of $\cO_{\bP(\cP_1)}(1)$ and $\pi^*\cO_{\p1}(1)$ respectively.  Trivially $\pi^*\cO_{\p1}(1)$ is globally generated. Since $\cP_1$ is globally generated, it follows that the same holds for $\cO_{F_1}(\xi_1)\cong\cO_{\bP(\cP_1)}(1)$: moreover, $\cO_{F_1}(\xi_1)\cong \sigma_1^*\cO_{\p3}(1)$. 

In both the cases we have an embedding $F_e\subseteq\p{29}$ induced by the linear system $\cO_{F_e}(h_e)=\cO_{F_e}(3\xi_{e}+(2-e)f)$ and $\omega_{F_e}\cong\cO_{F_e}(-h_e)$: in particular $F_e$ is a Fano threefold with $i_{F_e}=1$ and $\deg(F_e)=h_e^3=54$. 

If $e=1$, then let $H\subseteq\p3$ be a plane through $R$. On the one hand, $\sigma_1^{-1}(H)$ is in the class of $\xi_1$. On the other hand, $\sigma_1^{-1}(H)$ is the union of $E$ with the strict transform of $H$. Such a strict transform is in the linear system $\vert f\vert$, hence $E$ is the unique element in $\vert \xi_1-f\vert$. Notice that $Eh_1^2=6$.

Recall that  $\xi_e^3=e\xi_e^2 f$, and $\xi_e^2f$ is the class of a point. The morphism $\pi$ is smooth, hence we have the relative Euler exact sequence
\begin{equation}
\label{seqOmega}
0\longrightarrow\Omega_{F_e\vert \p1}\longrightarrow\cO_{F_e}(-\xi_e)^{\oplus2}\oplus\cO_{F_e}(-\xi_e+ef)\longrightarrow\cO_{F_e}\longrightarrow0.
\end{equation}
and the exact sequence of sheaves of differentials
\begin{equation*}
0\longrightarrow\cO_{F_e}(-2f)\longrightarrow\Omega_{F_e}\longrightarrow\Omega_{F_e\vert \p1}\longrightarrow0
\end{equation*}
A simple Chern class computation then yields $c_2(\Omega_{F_e})=3\xi_e^2+(6-2e)\xi_e f$. In particular, if $\cE$ is an instanton bundle with charge $\alpha\xi_e^2+\beta \xi_e f$ on $F_e$, then $c_1(\cE)=-h_e$, hence Equalities \eqref{RRgeneral} and \eqref{c_1c_2} yield
\begin{equation}
\label{RRFano}
\chi(\cE(a\xi_e+b f))=e\left(\frac{a^3}3+\frac{2a}3-a\alpha\right)+a^2b+3a+2b-b\alpha-a\beta.
\end{equation}

Notice that the pull--back via $\pi$ of the Euler sequence on $\p1$ returns the exact sequence
\begin{equation}
\label{seqEuler}
0\longrightarrow\cO_{F_e}(-f)\longrightarrow\cO_{F_e}^{\oplus2}\longrightarrow\cO_{F_e}(f)\longrightarrow0.
\end{equation}

We now describe three interesting families of smooth rational curves inside $F_e$.

\begin{remark}
\label{rLine}
Let $L$ be a line on $F_e$, i.e. a curve such that $Lh_e=1$. If we denote by $a\xi_e^2+b\xi_e f$ its class in $A^2(F_e)$, then we must have
$$
1=(a\xi_e^2+b\xi_e f)(3\xi_e+(2-e)f)=2(1+e)a+3b.
$$ 

Since $\cO_{F_e}(f)$ and $\cO_{F_e}(\xi_e)$ are globally generated, it follows that
\begin{equation}
\label{Positive}
a=Lf\ge0,\qquad ae+b=L\xi_e\ge0.
\end{equation}
Thus, $e=1$ necessarily.

If $b=LE\ge0$, then $a\le0$, hence $a=b=0$ necessarily. It follows that $b\le-1$, hence $0\le 4(a+b)=1+b\le0$ finally yields $a=1$ and $b=-1$, hence the class of $L$ is $\xi_1^2-\xi_1 f$. Notice that in this case $L\subseteq E$ because $L$ is integral and $LE=-1$. 

In particular $L$ is cut out on $E$ by a divisor in $\vert \xi_1\vert$. The cohomology of the exact sequence
$$
0\longrightarrow\cO_{F_1}(-\xi_1+f)\longrightarrow\cO_{F_1}\longrightarrow\cO_E\longrightarrow0
$$
tensored by $\cO_{F_1}(\xi_1)$, the isomorphism $\pi_*\cO_{F_1}(\xi_1)\cong\cP_1$ (see \cite[Exercise III.8.4 (a)]{Ha2}) and \cite[Exercises III.8.1 and III.8.3]{Ha2} imply that the linear system $\vert L\vert$ on the surface $E\cong\p1\times\p1$ has dimension $1$, hence $\vert L\vert$ is one of the rulings of lines on $E$: in particular distinct elements in $\vert L\vert$ do not intersect each other. The Hilbert scheme $\Lambda$ of lines inside ${F_1}$ is then isomorphic to $\p1$ and $\cO_L$ fits into the exact sequence
\begin{equation*}
0\longrightarrow\cO_{F_1}(-2\xi_1+f)\longrightarrow\cO_{F_1}(-\xi_1+f)\oplus\cO_{F_1}(-\xi_1)\longrightarrow\cO_{F_1}\longrightarrow\cO_L\longrightarrow0.
\end{equation*}
Restricting the above sequence to $L$ we finally obtain $\cN_{L\vert F_1}\cong\cO_{\p1}\oplus\cO_{\p1}(-1)$. 
Conversely, the intersection $L$ of general elements in $\vert \xi_1-f\vert$ and $\vert \xi_1\vert$ is a smooth curve. Since $Lh_1=1$, it follows that $L$ represents a point in $\Lambda$, thanks to the Bertini theorem.
\end{remark}

\begin{remark}
\label{rPseudoConic}
If $e=0$, let $M$ be a fibre of $\sigma_0$. If $e=1$ let $M$ be the pull--back of a line not intersecting the blown up line $R\subseteq\p3$. 

Trivially $M\cong\p1$, its class inside $A^2(F_e)$ is $\xi_e^2$ and we have $Mh_e=2(1+e)$. Consider now the very ample line bundle $\cO_{F_e}(\widehat{h}_e):=\cO_{F_e}(\xi_e+f)$: it is easy to check that $M\widehat{h}_e=1+e$. In what follows we will denote by $\Lambda_M$ the Hilbert scheme of curves in $F_e$ obtained as described above: $\Lambda_M$ is isomorphic to $\p2$ if $e=0$ and to an open set of the Grassmann variety of lines in $\p3$ if $e=1$, hence it is irreducible and rational of dimension $2(1+e)$. 

Notice that not all curves in the class $\xi_1^2\in A^2({F_1})$ represent a point in $\Lambda_M$: e.g. every union of a curve in $\Lambda$ with a curve with class $\xi_1f$ has class $\xi_1^2$. 

The structure sheaf $\cO_M$ fits into the exact sequence
\begin{equation*}
0\longrightarrow\cO_{F_e}(-2\xi_e)\longrightarrow\cO_{F_e}(-\xi_e)^{\oplus2}\longrightarrow\cO_{F_e}\longrightarrow\cO_M\longrightarrow0.
\end{equation*}
In particular we have $\cN_{M\vert {F_e}}\cong\cO_{\p1}(e)^{\oplus2}$.  
Conversely, the intersection $M$ of two general elements in $\vert \xi_e\vert$ is a smooth curve representing a point in $\Lambda_M$ by the Bertini theorem.

Clearly, distinct general elements in $\Lambda_M$ do not intersect each other and it is easy to check that they similarly do not intersect the general element in $\Lambda$.
\end{remark}

\begin{remark}
\label{rPseudoLine}
In the Remark \ref{rLine} we dealt with lines on $F_e$ embedded in $\p{29}$ via $\cO_{F_e}(h_e)$. 

It is easy to check that every line $L$ on $F_e$ also satisfies $L\widehat{h}_e=1$. Conversely, if $N$ is any curve with class $a\xi_e^2+b\xi_e f$ such that $N\widehat{h}_e=1$, then $N\cong\p1$, because $\cO_{F_e}(\widehat{h}_e)$ is very ample. Moreover, 
$$
1=(a\xi_e^2+b\xi_e f)(\xi_e+f)=(1+e)a+b,
$$
where $a$ and $ea+b$ are still non--negative, hence $0\le a\le1$ and $0\le ea+b\le1$. 

If $a=0$, then $b=1$, i.e. the class of $N$ is $\xi_e f$. If $a=1$, then $ea+b=0$. If $e=0$, then $b=0$, i.e. the class of $N$ is $\xi_e^2$: if $e=1$, then $b=-1$ and the class of $N$ is $\xi_1^2-\xi_1 f$. The latter case has been studied in the  Remark \ref{rLine}, while the former case has been described in Remark \ref{rPseudoConic}.

Let us deal with the case $a=0$ and $b=1$. To this purpose, we will denote by $\Lambda_N$ the Hilbert scheme of curves in $F_e$ whose class in $A^2(F_e)$ is $\xi_e f$. The equality $Nf=0$ implies that $N$ is contained in a fibre of $\pi$, hence $N$ is cut out on that fibre by a divisor in the linear system $\vert\xi_e\vert$. In particular  $\Lambda_N$ is dominated by a projective bundle on $\vert f\vert$ with fibre $\vert\xi_e\vert$, hence it is irreducible and rational of dimension $3$. 

The structure sheaf $\cO_N$ fits into the exact sequence
\begin{equation*}
0\longrightarrow\cO_{F_e}(-\xi_e-f)\longrightarrow\cO_{F_e}(-\xi_e)\oplus\cO_{F_e}(-f)\longrightarrow\cO_{F_e}\longrightarrow\cO_N\longrightarrow0.
\end{equation*}
In particular we have $\cN_{N\vert {F_e}}\cong\cO_{\p1}\oplus\cO_{\p1}(1)$. 
Conversely, the intersection $N$ of two general elements in $\vert \xi_e\vert$ and $\vert f\vert$ is a smooth curve representing a point in $\Lambda_N$ by the Bertini theorem.

Moreover, being both $\cO_{F_e}(\xi_e)$ and $\cO_{F_e}(f)$ are globally generated, we know that distinct general elements in $\Lambda_N$ do not intersect each other: for the same reason they do not intersect the general elements in $\Lambda$ and $\Lambda_M$.
\end{remark}

We close this section  by stating the following lemma which will also widely used  in the next sections.

\begin{lemma}
\label{lHoppe}
Let $\cG$ be a rank $2$ vector bundle on $F_e$.

Then $\cG$ is $\mu$--stable (resp. $\mu$--semistable) with respect to $\cO_{F_e}(h)$ if and only if  $h^0\big(F_e,\cG(-a\xi_e-b f)\big)=0$ for each $a,b \in\bZ$ such that $3(1+e)a+9(a+b) \ge\mu(\cG)$ (resp. $>\mu(\cG)$).
\end{lemma}
\begin{proof}
The group $\Pic(F_e)$ is generated by the classes of $\xi_e$ and $f$, hence it suffices to apply \cite[Corollary 4]{J--M--P--S}: see also \cite{Hop}.
\end{proof}

\section{Monadic description of instanton bundles on the blow up of $\p3$}
\label{sMonad}
In Sections \ref{sMonad}, \ref{sInstanton}, \ref{sEarnest} and \ref{sFinal} we deal with the blow up $F_1$ of $\p3$ along a line $R$. For this reason we will omit the $e=1$ in the subscripts, simply writing $F$, $\xi$, $\sigma$, $\cP$, $h$, $\cC^\bullet$. In this case $\xi^3=\xi^2f=1$ and $3(1+e)a+9(a+b)=15a+9b$ in Lemma \ref{lHoppe}.

 In this section we will construct a monad associated to each instanton bundle on $F$. In what follows we repeatedly need the cohomology of $\cO_F(a\xi+b f)$. We compute it in the  next proposition. 

\begin{proposition}
\label{pLineBundle}
We have
$$
\begin{aligned}
h^0\big(F,\cO_F(a\xi+b f)\big)&=\sum_{j=1}^{a+1}j{a+b+2-j\choose1},\\
h^1\big(F,\cO_F(a\xi+b f)\big)&=\sum_{j=1}^{a+1}j{-a-b-2+j\choose1},\\
h^2\big(F,\cO_F(a\xi+b f)\big)&=\sum_{j=1}^{-a-2}j{a+b+2+j\choose1},\\
h^3\big(F,\cO_F(a\xi+b f)\big)&=\sum_{j=1}^{-a-2}j{-a-b -2-j\choose1}
\end{aligned}
$$
where the summation is $0$ if the upper limit is smaller than the lower limit.
\end{proposition}
\begin{proof}
On the one hand, if $a\ge-1$, then \cite[Exercises III.8.1, III.8.3 and III.8.4]{Ha2} implies that
$$
h^i\big(F,\cO_F(a\xi+b f)\big)=h^i\big(\p1,\cO_{\p1}(b )\otimes\pi_*\cO_F(a\xi)\big)=\sum_{j=1}^{a+1}h^i\big(\p1,\cO_{\p1}(a+b +1-j)^{\oplus j}\big).
$$
On the other hand, if $a\le-1$, then Equality \eqref{Serre} yields
$$
h^i\big(F,\cO_F(a\xi+b f)\big)=h^{3-i}\big(F,\cO_F(-(a+3)\xi-(b +1)f)\big).
$$
The statement then follows by combining the above equalities.
\end{proof}

A trivial consequence of the above proposition is that $\cO_F(a\xi+b f)$ is an effective line bundle if and only if $a,a+b\ge0$. 

Recall that $\Mov(F)\subseteq A^2(F)$ is the dual of the pseudo--effective cone of $F$, i.e. it is the closure inside $A^2(F)$ of the set of cycles $\zeta\in A^2(F)$ such that $\zeta D\ge0$ for each effective divisor $D\subseteq F$: for further details on $\Mov(X)$ see \cite[Section 11,4.C]{Laz2}).

\begin{corollary}
\label{cPositive}
The cycle $\alpha\xi^2+\beta \xi f\in A^2(F)$ is in $\Mov(F)$ if and only if $\alpha,\beta\ge0$. 
\end{corollary}
\begin{proof}
The pseudo--effective cone is generated by the effective divisor $a\xi+ bf$, i.e. such that  $a,a+b\ge0$. Thus the equality
$$
(\alpha\xi^2+\beta\xi f)(a\xi+bf)=\alpha(a+b)+\beta a,
$$
implies $\alpha\xi^2+\beta \xi f\in \Mov(F)$ if and  only if if and $\alpha,\beta\ge0$, which is trivial.
\end{proof}

Consider the following ordered sets of vector bundles on $F$
\begin{align*}
(\cF_{-5},\cF_{-4}&,\cF_{-3},\cF_{-2},\cF_{-1},\cF_0):=\\
:=&(\cO_F(-\xi),\cO_F(-\xi+f),\cO_F(-f),\cO_F,\cO_F(\xi-2f),\cO_F(\xi-f)),
\end{align*}
\begin{align*}
(\cG_{0},\cG_{1}&,\cG_{2},\cG_{3},\cG_{4},\cG_5):=\\
:=&(\cO_F(-\xi+f),\cO_F(-\xi),\Omega_{F\vert\p1}^1,\Omega_{F\vert\p1}^1(-f), \cO_F(-2\xi),\cO_F(-2\xi-f)).
\end{align*}
(these are the Orlov collection with respect to $\cO_F(\xi-f)$ and its dual tensored by $\cO_F(\xi-f)$ and $\cO_F(-\xi+f)$ respectively: see \cite[Corollary 2.6]{Orl}). 

\begin{lemma}
\label{lAO}
Let $\cE$ be an instanton bundle on $F$.

Then $\cE$  is the cohomology in degree $0$ of a complex $\widehat{\cC}^\bullet$ with $i^{th}$--module
$$
\widehat{\cC}^i:= \bigoplus_{q+p=i}H^{q+\lceil \frac p2\rceil}\big(F,\cE\otimes\cF_{p}\big)\otimes\cG_{-p}.
$$
\end{lemma}
\begin{proof}
Recall that $F\cong\bP(\cP)$, hence we can apply \cite[Theorem 8]{A--O}: notice that, with the notation in that paper, $\bP(\mathcal H):=\bP(\mathrm{Sym}(\mathcal H^\vee))$. 

In our case, we have $\mathcal H=\cP(-1)\cong\cO_{\p1}\oplus\cO_{\p1}(-1)^{\oplus2}$: in order to apply \cite[Theorem 8]{A--O} we must consider $\mathcal H(1)$, hence the relative universal line bundle therein (i.e. the tautological line bundle of $\pi$) is $\cO_F(f-\xi)$. The relative universal quotient bundle $\cQ$ can be computed by dualizing Sequence \eqref{seqOmega}, hence $\cQ^\vee\cong\Omega_{F\vert\p1}(\xi-f)$.

Recall that there is a natural functor $\cA\mapsto\cA^\bullet$ from the category of coherent sheaves on $F$ to the category of complexes of coherent sheaves on $F$, where
$$
\cA^i=\left\lbrace\begin{array}{ll} 
0\quad&\text{if $i\ne 0$,}\\
\cA\quad&\text{if $i=0$.}
\end{array}\right.
$$

In particular, \cite[Theorem 8]{A--O} applied to $\cE(\xi-f)^\bullet$ yields that it is the cohomology of a complex with 
\begin{align*}
\bigoplus_{s+p=i} \bigoplus_{a+b=p}H^{s}\big(F,\cE(&(a+1)\xi+(b-a-1)f)\big)\otimes\\
&\otimes\wedge^{-a}(\Omega_{F\vert\p1}(\xi-f))\otimes\pi^*\wedge^{-b}(\Omega_{\p1}(1))
\end{align*}
in degree $i$. It turns out that such a complex is everywhere exact, but in degree $0$ where its cohomology is exactly $\cE(\xi-f)$. Thus the definitions of $\cF_p$, $\cG_p$ and simple computations lead to the statement.
\end{proof}

We deduce from the above statement that in order to prove Theorem \ref{tSimplify} we have to compute the cohomologies $e^{p,q}:=h^{q+\lceil\frac p2\rceil}\big(F,\cE\otimes\cF_{p}\big)$ for $0\le q\le 5$ and $-5\le p\le 0$.

\begin{proposition}
\label{pTable}
Let $\cE$ be an instanton bundle on $F$.

If $c_2(\cE)=\alpha\xi^2+\beta \xi f$ and 
$$
\gamma:=h^1\big(F,\cE(-\xi+f)\big),\qquad \delta:= h^1\big(F,\cE(-\xi+2f)\big)
$$
then $e^{p,q}$ is the number in position $(p,q)$ in the following table.
\begin{table}[H]
\centering
\bgroup
\def\arraystretch{1.5}
\begin{tabular}{ccccccc}
\cline{1-6}
\multicolumn{1}{|c|}{0} & \multicolumn{1}{c|}{0} & \multicolumn{1}{c|}{0} & \multicolumn{1}{c|}{0} & \multicolumn{1}{c|}{0} & \multicolumn{1}{c|}{0} & $q=5$ \\ \cline{1-6}
\multicolumn{1}{|c|}{$\alpha+\beta-4$} & \multicolumn{1}{c|}{$\beta+\gamma-1$} & \multicolumn{1}{c|}{0} & \multicolumn{1}{c|}{0} & \multicolumn{1}{c|}{0} & \multicolumn{1}{c|}{0} & $q=4$ \\ \cline{1-6}
\multicolumn{1}{|c|}{0} & \multicolumn{1}{c|}{$\gamma$} & \multicolumn{1}{c|}{$\alpha-2$} & \multicolumn{1}{c|}{0} & \multicolumn{1}{c|}{0} & \multicolumn{1}{c|}{0} & $q=3$ \\ \cline{1-6}
\multicolumn{1}{|c|}{0} & \multicolumn{1}{c|}{0} & \multicolumn{1}{c|}{0} & \multicolumn{1}{c|}{0} & \multicolumn{1}{c|}{$\delta$} & \multicolumn{1}{c|}{$\gamma$} & $q=2$ \\ \cline{1-6}
\multicolumn{1}{|c|}{0} & \multicolumn{1}{c|}{0} & \multicolumn{1}{c|}{0} & \multicolumn{1}{c|}{0} & \multicolumn{1}{c|}{$\delta+\beta-\alpha+2$} & \multicolumn{1}{c|}{$\beta+\gamma-1$}& $q=1$ \\ \cline{1-6}
\multicolumn{1}{|c|}{0} & \multicolumn{1}{c|}{0} & \multicolumn{1}{c|}{0} & \multicolumn{1}{c|}{0} & \multicolumn{1}{c|}{0} & \multicolumn{1}{c|}{0} & $q=0$ \\ \cline{1-6}
$p=-5$ & $p=-4$ & $p=-3$ & $p=-2$ & $p=-1$ & $p=0$
\end{tabular}
\egroup
\caption{The values of $e^{p,q}$}
\end{table}
\end{proposition}
\begin{proof}
By definition $e^{p,q}=0$ for $p\le -2$ and $q=0$, $p\le -4$ and $q=1$, $p\ge-1$ and $q=4$, $p\ge-3$ and $q=5$.

The vanishings  $h^s\big(F,\cE\otimes\cF_{p}\big)=0$ for $s=0$ and each $p$ follow from Lemma \ref{lHoppe} because $\cE$ is $\mu$--semistable. The same argument and Equality \eqref{Serre} yield the vanishings also for $s=3$ and each $p$. Thus $e^{p,q}=0$ also for $p=0,-1$ and $q=0,3$, $p=-2,-3$ and $q=1,4$, $p=-4,-5$ and $q=2,5$.

Lemma \ref{lNatural} yields $e^{-2,q}=h^{q-1}\big(F,\cE\otimes\cF_{-2}\big)=0$ for $q=2,3$. Thanks to such a vanishing for $q=2$, the cohomology of Sequence \eqref{seqEuler} yields
$$
e^{-3,2}=h^1\big(F,\cE\otimes\cF_{-3}\big)\le h^0\big(F,\cE(f)\big)=0,
$$
thanks to Lemma \ref{lHoppe}.

The cohomology of Sequence \eqref{seqOmega} and its dual tensored by $\cE(\xi)$ and $\cE(-2\xi+f)$ respectively, the vanishings proved above and Equality \eqref{Serre1} yield
\begin{align*}
e^{-5,3}=h^1\big(F,\cE\otimes\cF_{-5}\big)&=h^1\big(F,\cE(-\xi)\big)=h^2\big(F,\cE(\xi)\big)\le h^3\big(F,\cE\otimes\Omega_{F\vert \p1}(\xi)\big)=0.
\end{align*}  

By definition $e^{-4,3}=h^1\big(F,\cE\otimes\cF_{-4}\big)=\gamma$  and $e^{-1,2}=h^2\big(F,\cE\otimes\cF_{-1}\big)=\delta$. Equality \eqref{Serre1} then also returns $e^{0,2}=h^2\big(F,\cE\otimes\cF_0\big)=\gamma$. All the remaining values of $e^{p,q}=h^{q+\lceil\frac p2\rceil}\big(F,\cE\otimes\cF_{p}\big)$ are computed by means of Equality \eqref{RRFano}.

The statement is then completely proved.
\end{proof}

Proposition \ref{pTable} has some interesting consequences for an instanton bundle on $F$.

\begin{corollary}
\label{cBound}
Let $\cE$ be an instanton bundle with $c_2(\cE)=\alpha\xi^2+\beta  \xi f$ on $F$.

Then $\alpha\ge2$ and 
\begin{gather*}
\beta\ge \max\left\{\ 4-\alpha,\ \alpha-2-\delta,\ 1-\gamma\ \right\},\\
\alpha+\beta-4+2\gamma\ge\delta\ge 2\gamma.
\end{gather*}
\end{corollary}
\begin{proof}
All the inequalities follow from the obvious non--negativity of the $e_{p,q}$'s, but the last line which is obtained by computing the cohomology of  Sequence \eqref{seqEuler} tensored by $\cE(-\xi+f)$.
\end{proof}

Secondly, $c_2(\cE)h\ge27/2$ (see Inequality \eqref{Minimal}). Thus, $14$ is the first integral value that $c_2(\cE)h$ could attain. Proposition \ref{pTable} allows us to give the following sharper lower bound on the degree of the charge.

\begin{corollary}
\label{cMinimal}
If $\cE$ is an instanton bundle on $F$ with $c_2(\cE)=\alpha\xi^2+\beta  \xi f$, then $c_2(\cE)h=4\alpha+3\beta\ge15$.
\end{corollary}
\begin{proof}
Notice that $\alpha\ge2$ and $\alpha+\beta\ge4$ (see Corollary \eqref{cBound}) and $c_2(\cE)h\ge14$. If equality occurs and  $\alpha\ge3$, then
$$
4\le\alpha+\beta=\frac{14-\alpha}3<4,
$$
a contradiction. Thus, we deduce $\alpha=2$. The same argument used above yields $\alpha=\beta=2$ necessarily, hence $\delta=2\gamma$, thanks to Corollary \eqref{cBound}. 

The cohomology of Sequence \eqref{seqEuler} tensored by $\cE(\xi)$, Equality \eqref{Serre1} and the equality $h^1\big(F,\cE(\xi-f)\big)=\gamma+1$ (see the computation of $e^{0,1}$ in the proof of Proposition \ref{pTable}) imply $h^0\big(F,\cE(\xi+f)\big)=\gamma+1$.
Thus $\cE$ is not $\mu$--semistable thanks to Lemma \ref{lHoppe}, hence it is not an instanton bundle.
\end{proof}

The following remark will be helpful for proving Theorem \ref{tSimplify} stated in the introduction.

\begin{remark}
\label{rSmooth}
We show that $\vert a\xi+bf\vert$ contains a smooth integral divisor $D$ if and only if either $a,b\ge0$, or $a=-b=1$.

To this purpose we first notice that $\cO_F(a\xi+b f)$ is globally generated if and only if $a,b\ge0$. Indeed, on the one hand, if $a,b\ge0$ the assertion is a trivial consequence of the existence of a surjective morphism $\pi^*\cP\to\cO_F(\xi)$. On the other hand, if $\cO_F(a\xi+b f)$ is globally generated, then $a=(a\xi+b f)\xi f$ and $b=(a\xi+b f)(\xi^2-\xi f)$ must be non--negative.

If $a=-b=1$, then $D=E\cong\p1\times\p1$ which is trivially smooth and integral. If $a,b\ge0$, then $\cO_F(a\xi+b f)$ is globally generated, hence $\vert a\xi+bf\vert$ contains a smooth integral divisor thanks to the Bertini theorem. 

Conversely, assume that $\vert a\xi+bf\vert$ contains a smooth integral divisor. Thus if $\cO_F(a\xi+b f)$ is not globally generated, then $a\ge1$ and $-1\ge b\ge-a$, thanks to  Proposition \ref{pLineBundle}.
If $E\not\subseteq D$, then there is a line $L\subseteq E$ intersecting $D$ properly, hence $0\le DL=b\le-1$, a contradiction. Thus $E\subseteq D$ which is smooth and integral, hence $D=E$.
\end{remark}

Also thanks to Proposition \ref{pTable} we can prove Theorem \ref{tSimplify} stated in the introduction. 

\medbreak
\noindent{\it Proof of Theorem \ref{tSimplify}.}
By applying Lemma \ref{lAO} using the values $h^q\big(F,\cE\otimes\cF_{-p})$ calculated in Proposition \ref{pTable} we obtain a complex $\widehat{\cC}^\bullet$ where
\begin{gather*}
\widehat{\cC}^{-1}:=\cO_F(-2\xi-f)^{\oplus\alpha+\beta-4}\oplus\cO_F(-2\xi)^{\oplus\gamma},\\
\widehat{\cC}^0:=\cO_F(-2\xi)^{\oplus\beta+\gamma-1}\oplus\Omega_{F\vert\p1}(-f)^{\oplus\alpha-2}\oplus\cO_F(-\xi)^{\oplus\delta+\beta-\alpha+2},\\
\widehat{\cC}^1:=\cO_F(-\xi)^{\oplus\delta}\oplus\cO_F(-\xi+f)^{\oplus\beta+\gamma-1},\\
\widehat{\cC}^2:=\cO_F(-\xi+f)^{\oplus\gamma},
\end{gather*}
which is exact everywhere but at $\widehat{\cC}^0$ where its cohomology is $\cE$. Notice that $\cC^i\cong\widehat{\cC}^{i}$ for $i=-1,0$: thus the statement is proved if we check that $\cC^1$ is isomorphic to the kernel of the differential $\widehat{\cC}^1\to\widehat{\cC}^2$.

Let $\varphi$ and $\psi$ be the differentials $\widehat{\cC}^1\to\widehat{\cC}^2$ and $\widehat{\cC}^0\to\widehat{\cC}^1$ twisted by the identity of $\cO_F(\xi-f)$. 

We have $\cO_F(-f)\cong\pi^*\cO_{\p1}(-1)$ and $\cO_F\cong\pi^*\cO_{\p1}$, hence \cite[Exercise III.8.3]{Ha2} implies $R^i\pi_*\cO_F=R^i\pi_*\cO_F(-f)=0$. The functor $\pi_*$ then induces an isomorphism
$$
\theta\colon \Hom_F\big(\widehat{\cC}^1(\xi-f),\widehat{\cC}^2(\xi-f)\big)\longrightarrow
\Hom_{\p1}\big(\cO_{\p1}(-1)^{\oplus\delta}\oplus\cO_{\p1}^{\oplus\beta+\gamma-1},\cO_{\p1}^{\oplus\gamma}\big)
$$
thanks to the projection formula (see \cite[Exercise III.8.1]{Ha2}, where we are using that $\Hom_X\big(\cdot,\cdot\big)$ are the global sections of $\sHom_X\big(\cdot,\cdot\big)$). Let $\theta(\varphi)=\phi$: if $\phi$ is not surjective at $x\in \p1$, then $\varphi$ is not surjective at the points of $\pi^{-1}(x)$. It follows that $\phi$ is surjective, hence
$$
\ker(\varphi)\cong\bigoplus_{i=1}^{\delta+\beta-1}\cO_{F}(-\lambda_if),
$$ 
for suitable integers $\lambda_i$. Since
$$
\ker(\varphi)\subseteq\widehat{\cC}^1(\xi-f):=\cO_F(-f)^{\oplus\delta}\oplus\cO_F^{\oplus\beta+\gamma-1},
$$
it follows that $\lambda_i\ge0$.

By composing $\psi$ with the projections on the summands of $\ker(\varphi)=\im(\psi)$, we obtain epimorphisms $\psi_i\colon \widehat{\cC}^0\to\cO_F(-\lambda_if)$. We have
\begin{gather*}
\Hom_F\big(\cO_F(-f),\cO_F(-\lambda_if)\big)=H^0\big(F,\cO_F((1-\lambda_i)f)\big),\\
\Hom_F\big(\cO_F(-\xi-f),\cO_F(-\lambda_if)\big)=H^0\big(F,\cO_F(\xi+(1-\lambda_i)f)\big),
\end{gather*}
Thanks to Proposition \ref{pLineBundle} it is easy to check that the first space vanishes if $\lambda_i\ge2$ and that the same is true for the second one when $\lambda_i\ge3$. By applying $\Hom_F\big(\cdot,\cO_F(-\lambda_if)\big)$ to Sequence \eqref{seqOmega} one also deduces that 
$$
\Hom_F\big(\Omega_{F\vert\p1}(\xi-2f),\cO_F(-\lambda_if)\big)=0
$$
if $\lambda_i\ge3$. In particular $\psi_i$ cannot be surjective when $\lambda_i\ge3$, hence we deduce 
$$
\ker(\varphi)\cong\cO_F(-2f)^{\oplus\varepsilon}\oplus\cO_F(-f)^{\oplus\eta}\oplus\cO_F^{\oplus\beta+\delta-\eta-\varepsilon-1}.
$$
By computing the cohomology of the exact sequence
$$
0\longrightarrow\ker(\varphi)\longrightarrow\widehat{\cC}^1(\xi-f)\longrightarrow\widehat{\cC}^2(\xi-f)\longrightarrow0,
$$
we finally deduce that $\eta=\delta-2\varepsilon$, i.e.
$$
\ker(\varphi)\cong\cO_F(-2f)^{\oplus\varepsilon}\oplus\cO_F(-f)^{\oplus\delta-2\varepsilon}\oplus\cO_F^{\oplus\beta+\varepsilon-1}.
$$

Let $\cC^{-1}:=\widehat{\cC}^{-1}$, $\cC^0:=\widehat{\cC}^{0}$ and $\cC^{1}:=\ker(\varphi)\otimes\cO_F(-\xi+f)$. We have then a monad $\cC^\bullet$ whose cohomology is $\cE$. In order to complete the proof of the first part of the statement it suffices to check that $\varepsilon=\gamma$. To this purpose consider the two short exact sequences
\begin{equation}
\label{Display}
\begin{gathered}
0\longrightarrow \cK\longrightarrow \cC^0\longrightarrow\cC^{1}\longrightarrow0,\\
0\longrightarrow \cC^{-1}\longrightarrow \cK\longrightarrow\cE\longrightarrow0.
\end{gathered}
\end{equation}
Proposition \ref{pLineBundle}  and the cohomology of the dual of Sequence \eqref{seqOmega}  tensored by $\cO_F(-2\xi-f)$ yield $h^i\big(F,\Omega_{F\vert\p1}(\xi-2f)\big)=0$ for $i=1,2$. Thus the cohomology of the above Sequences \eqref{Display} tensored by $\cO_F(\xi-f)$ and Equality \eqref{Serre1} finally returns
$\varepsilon=h^2\big(F,\cE(\xi-f)\big)=\gamma$. 

Conversely, assume that the cohomology $\cE$ of Monad \eqref{Monad}  is a $\mu$--semistable vector bundle of rank $2$ (so that  $h^0\big(F,\cE\big)=0$ as pointed out in the introduction). 

Easy and tedious computations lead to the equalities
\begin{gather*}
c_1(\cE)=c_1(\cC^0)-c_1(\cC^1)-c_1(\cC^{-1})=-3\xi-f,\\
\begin{align*}
c_2(\cE)&=c_2(\cC^0)-c_2(\cC^1)-c_2(\cC^{-1})-c_1(\cC^0)c_1(\cC^{-1})-c_1(\cC^0)c_1(\cC^{1})+\\
&+c_1(\cC^{-1})^2+c_1(\cC^{-1})c_1(\cC^{1})+c_1(\cC^{1})^2=\alpha\xi^2+\beta f^2.
\end{align*}
\end{gather*}
Moreover, we can still consider Sequences \eqref{Display} which easily lead to the inequality 
\begin{equation}
\label{BoundDisplay}
h^i\big(F,\cE\otimes\mathcal L\big)\le \sum_{j=-1}^1h^{i-j}\big(F,\cC^{j}\otimes\mathcal L\big)
\end{equation}
for each $\mathcal L\in\Pic(F)$.

Let $D$ be either $0$, or any smooth element in $\vert a\xi+bf\vert$, $D\ne E$: thanks to Remark \ref{rSmooth} we then know that $a,b\ge0$. Thanks to Proposition \ref{pLineBundle} and the cohomology of Sequence \eqref{seqOmega} tensored by $\cO_F(-a\xi-bf)$, Inequality \eqref{BoundDisplay} with 
$$
\mathcal L:=\cO_F(-D)\cong\cO_F(-a\xi-bf)
$$
finally yields $h^1\big(F,\cE(-D)\big)=0$. If $D=0$, then we deduce that $\cE$ satisfies the instantonic condition, hence it is an instanton, because it is assumed $\mu$--semistable. If $D\ne0$, we obtain the assertion (3) of the statement.

Proposition \ref{pLineBundle}  and the cohomology of the dual of Sequence \eqref{seqOmega}  tensored by $\cO_F(-2\xi-df)$ yield $h^i\big(F,\Omega_{F\vert\p1}(\xi-f-df)\big)=0$ for $i=1,2$.
Thus, assertions (1) and (2) can be obtained by computing the cohomology of sequences \eqref{Display} tensored by $\cO_F(\xi-df)$ respectively, because $h^1\big(F,\cE(-\xi+df)\big)=h^2\big(F,\cE(\xi-df)\big)$ thanks to Equality \eqref{Serre1}, where $d=1,2$.
\qed
\medbreak

\begin{remark}
\label{rReferee1}
It is natural to ask if the required $\mu$--semistability of the cohomology $\cE$ of the monad $\cC^\bullet$ in the second part of the statement of Theorem \ref{tSimplify} is actually necessary for proving that $\cE$ is an instanton, or if it can be at least relaxed.

E.g., one could wonder if it can be replaced by the weaker vanishing $h^0\big(F,\cE\big)=0$, as in the statement of \cite[Theorem 4.2]{M--M--PL}. The $\mu$--semistability of $\cE$ has been used in the proof of Proposition \ref{pTable} (and hence in the construction of Monad \eqref{Monad}) in order to get the vanishings $h^0\big(F,\cE(\xi-f)\big)=h^0\big(F,\cE(\xi-2f)\big)=0$ which do not seem to follow from the vanishing of $h^0\big(F,\cE\big)$.

Indeed, let us consider a morphism $\varphi\colon \cO_F(-2\xi)\oplus\cO_F(-\xi)^{\oplus2}\to\cO_F(-\xi+f)$ with matrix
$$
A:=\left(\begin{array}{ccc}
0&a_1&a_2
\end{array}\right),
$$
where $a_1,a_2\in H^0\big(F,\cO_F(f)\big)$ have no common zeros. Thus $\varphi$ is surjective, hence it defines a monad $\Phi^\bullet$ coinciding with Monad \eqref{Monad} when $\alpha=\beta=2$ and $\gamma=\delta=0$.

Taking into account of the definition of $\varphi$ and of Sequence \eqref{seqEuler}, we deduce that the cohomology of $\Phi^\bullet$ is $\cE\cong\ker(\varphi)\cong\cO_F(-2\xi)\oplus\cO_F(-\xi-f)$. Thus $\cE$ is not $\mu$--semistable, because $\mu(\cO_F(-2\xi))=-30\ne-24=\mu(\cO_F(-\xi-f))$. In particular, $\cE$ is not an instanton bundle, though $h^0\big(F,\cE\big)=0$.
\end{remark}

\begin{remark}
\label{rEarnest}
If $\cE$ is earnest, then Monad \eqref{Monad} becomes
\begin{align*}
0&\longrightarrow \cO_F(-2\xi-f)^{\oplus\alpha+\beta-4}\longrightarrow\\
&\phantom{\longrightarrow }\longrightarrow  \cO_F(-2\xi)^{\oplus\beta-1}\oplus\Omega_{F\vert\p1}(-f)^{\oplus\alpha-2}\oplus\cO_F(-\xi)^{\oplus\delta+\beta-\alpha+2}\longrightarrow\\
&\phantom{\longrightarrow \longrightarrow }\longrightarrow \cO_F(-\xi)^{\oplus\delta}\oplus\cO_F(-\xi+f)^{\oplus\beta-1}\longrightarrow 0.
\end{align*}
\end{remark}

The following corollary is an immediate consequence of Theorem \ref{tSimplify} and Corollary \ref{cPositive}.

\begin{corollary}
\label{cSimplify}
Let $\cE$ be an instanton on $F$. Then $\cE$ is earnest if and only if
$$
h^1\big(F,\cE(-\xi+f)\big)=0.
$$
If this is true, then $c_2(\cE)\in{\Mov}(F)$,
\end{corollary}

\section{Existence of instanton bundles on the blow up of $\p3$}
\label{sInstanton}
In this section we will prove the existence of instanton bundles  satisfying some extra important properties for all the admissible charges. Again $\xi$ and $F$ denote $\xi_1$ and $F_1$ respectively.

\begin{construction}
\label{conInstanton}
Let $\alpha$ and $\beta$ be integers such that $\alpha\ge2$, $\alpha+\beta\ge4$ and $4\alpha+3\beta\ge15$. We take $L_1,\dots, L_{\alpha-2}$ and $N_1,\dots,N_{\alpha+\beta-4}$ pairwise disjoint curves corresponding to points in $\Lambda$ and $\Lambda_N$ respectively and define 
\begin{equation}
\label{ZInstanton}
Z:=\bigcup_{i=1}^{\alpha-2}L_i\cup\bigcup_{j=1}^{\alpha+\beta-4}N_j\subseteq F.
\end{equation}
If $\alpha=2$ and $\alpha+\beta=4$, then $4\alpha+3\beta=14$, hence the condition $4\alpha+3\beta\ge15$ implies $Z\ne\emptyset$.
As pointed out in Remarks \ref{rLine} and \ref{rPseudoLine}, both $L_i$ and $N_j$ are isomorphic to $\p1$. 

We claim that $\det(\cN_{Z\vert F})\cong\cO_F(\xi-f)\otimes\cO_Z$. Such an isomorphism can be checked component by component. The aformentioned remarks show that 
\begin{gather*}
\det(\cN_{Z\vert F})\otimes\cO_{L_i}\cong\cO_{\p1}(-1)\cong\cO_F(\xi-f)\otimes\cO_{L_i},\\
\det(\cN_{Z\vert F})\otimes\cO_{N_j}\cong\cO_{\p1}(1)\cong\cO_F(\xi-f)\otimes\cO_{N_j}.
\end{gather*}
Since we have $h^2\big(F,\cO_F(-\xi+f)\big)=0$ thanks to Proposition \ref{pLineBundle}, it follows from Theorem \ref{tSerre} the existence of a vector bundle $\cF$ on $F$ with a section $s$ vanishing exactly along $Z$ and with $c_1(\cF)=\xi-f$, $c_2(\cF)=Z$. 

Sequence \eqref{seqSerre} for such an $s$ tensored by $\cO_F(-2\xi)$ gives the exact sequence
\begin{equation}
\label{seqStandard}
0\longrightarrow\cO_F(-2\xi)\longrightarrow\cE\longrightarrow\cI_{Z\vert F}(-\xi-f)\longrightarrow0,
\end{equation}
where $\cE:=\cF(-2\xi)$.

Since $h^1\big(F,\cO_F(-\xi+f)\big)=0$, it follows that the bundle $\cE$ is uniquely determined by the scheme $Z$. 
\end{construction}

The main result of the section is the following proof of Theorem \ref{tInstanton} stated in the introduction.

\medbreak
\noindent{\it Proof of Theorem \ref{tInstanton}.}
We trivially have $c_1(\cE)=-h$ and $c_2(\cE)=\alpha\xi^2+\beta\xi f$ by construction. Moreover, $h^1\big(F,\cE\big)= h^1\big(F,\cI_{Z\vert F}(-\xi-f)\big)$ from the cohomology of Sequence \eqref{seqStandard}. 

For each connected component $Y\cong\p1$ of $Z$ we have $(-\xi-f)Y=-1$, hence $h^0\big(Z,\cO_F(-\xi-f)\otimes\cO_Z\big)=0$. The cohomology of the exact sequence
\begin{equation}
\label{seqIdeal}
0\longrightarrow\cI_{Z\vert F}\longrightarrow\cO_F\longrightarrow\cO_Z\longrightarrow0
\end{equation}
tensored by $\cO_F(-\xi-f)$ then yields $h^1\big(F,\cE\big)= h^1\big(F,\cI_{Z\vert F}(-\xi-f)\big)=0$.

We will now show that $\cE$ is $\mu$--stable. To this purpose we will make use of Lemma \ref{lHoppe}, proving that  if $15a+9b=\mu(\cO_F(a\xi+bf))\ge\mu(\cE)=-27$, i.e.
\begin{equation}
\label{AB}
b\ge-3-\frac53a
\end{equation}
then the cohomology of Sequence \eqref{seqStandard} tensored by $\cO_F(-a\xi-b f)$, i.e.
$$
0\longrightarrow\cO_F(-(a+2)\xi-bf)\longrightarrow\cE(-a\xi-bf)\longrightarrow\cI_{Z\vert F}(-(a+1)\xi-(b+1)f)\longrightarrow0,
$$
returns $h^0\big(F,\cE(-a\xi-b f)\big)=0$. If $a\ge0$ such a vanishing is trivial, hence we restrict our attention to the case $a\le-1$. 

If $a=-1$, then
$$
h^0\big(F,\cO_F(-(a+2)\xi-bf)\big)=h^0\big(F,\cO_F(-\xi-bf)\big)=0.
$$
Moreover, Inequality \eqref{AB} implies $b\ge-1$, hence 
$$
h^0\big(F,\cI_{Z\vert F}(-(a+1)\xi-(b+1)f)\big)= h^0\big(F,\cI_{Z\vert F}(-(b+1)f)\big)=0,
$$
because $Z\ne\emptyset$.

If $a\le-2$, then Inequality\eqref{AB} yields $-(a+2)-b=-(a+1)-(b+1)\le-1$, hence again 
\begin{gather*}
h^0\big(F,\cO_F(-(a+2)\xi-bf)\big)=0,\\
h^0\big(F,\cI_{Z\vert F}(-(a+1)\xi-(b+1)f)\big)\le h^0\big(F,\cO_F(-(a+1)\xi-(b+1)f)\big)=0.
\end{gather*}

We now prove that $\cE$ is generically trivial.  Indeed, if we restrict Sequence \eqref{seqStandard} to a line $L\in\Lambda$ not intersecting $Z$, one easily obtains the exact sequence
$$
0\longrightarrow\cO_{\p1}\longrightarrow\cE\otimes\cO_L\longrightarrow\cO_{\p1}(-1)\longrightarrow0,
$$
hence $\cE\otimes\cO_L\cong\cO_{\p1}\oplus\cO_{\p1}(-1)$ for such lines.

We now prove the assertion on the dimensions of the $\Ext$ groups. Since $\cE$ is $\mu$--stable, then it is simple, hence the equality $\Ext^3_{F}\big(\cE,\cE\big)=0$ follows from Equality \eqref{Ext3}. We will show below that
$$
\Ext^2_{F}\big(\cE,\cE\big)\cong H^2\big(F,\cE\otimes\cE^\vee\big)=0,
$$
hence 
$$
\dim\Ext^1_{F}\big(\cE,\cE\big)=8\alpha+6\beta-30
$$
thanks to Equality \eqref{Ext12}.

To this purpose, the cohomology of Sequence \eqref{seqStandard} tensored by $\cE^\vee\cong\cE(h)$ returns
$$
H^2\big(F,\cE(\xi+f)\big)\longrightarrow H^2\big(F,\cE\otimes\cE^\vee\big)\longrightarrow H^2\big(F,\cE\otimes\cI_{Z\vert F}(2\xi)\big),
$$
hence it suffices to check that $h^2\big(F,\cE(\xi+f)\big)=h^2\big(F,\cE\otimes\cI_{Z\vert F}(2\xi)\big)=0$.

We first check that $h^2\big(F,\cE(\xi+f)\big)=0$. Indeed, thanks to Proposition \ref{pLineBundle} the cohomologies of Sequences \eqref{seqStandard} tensored by $\cO_F(\xi+f)$ and \eqref{seqIdeal} return
$$
h^2\big(F,\cE(\xi+f)\big)\le h^1\big(F,\cO_Z\big).
$$
The dimension on the right is zero, because $Z$ is the disjoint union of smooth rational curves. 

Finally we check that $h^2\big(F,\cE\otimes\cI_{Z\vert F}(2\xi)\big)=0$. Thanks to Proposition \ref{pLineBundle}, the cohomology of Sequence \eqref{seqIdeal} tensored by $\cO_F(\xi-f)$ then yields
$$
h^2\big(F,\cI_{Z\vert F}(\xi-f)\big)\le h^1\big(Z,\cO_F(\xi-f)\otimes\cO_Z\big).
$$
Since $\cO_F(\xi-f)$ restricts to each component of $Z$ to a line bundle of degree either $1$ (if the component is in $\Lambda_N$), or $-1$ (if the component is in $\Lambda$), it follows that the dimension on the right is zero. In particular $h^2\big(F,\cI_{Z\vert F}(\xi-f)\big)=0$, hence the cohomology of Sequence \eqref{seqStandard} tensored by $\cO_F(2\xi)$ and Proposition \ref{pLineBundle} imply $h^2\big(F,\cE(2\xi)\big)=0$. We deduce that the cohomology of Sequence \eqref{seqIdeal} tensored by $\cE(2\xi)$ returns 
\begin{align*}
h^2\big(F,\cE\otimes\cI_{Z\vert F}(2\xi)\big)&\le h^1\big(Z,\cE(2\xi)\otimes\cO_Z\big)=\\
&=\sum_{i=1}^{\alpha-2}h^1\big(Z,\cE(2\xi)\otimes\cO_{L_i}\big)+\sum_{j=1}^{\alpha+\beta-4}h^1\big(Z,\cE(2\xi)\otimes\cO_{N_j}\big).
\end{align*}
Equality \eqref{Normal} and the definition of $\cE$ imply $\cE(2\xi)\otimes\cO_Z\cong\cN_{Z\vert F}$. Thus
\begin{gather*}
\cE(2\xi)\otimes\cO_{L_i}\cong\cO_{\p1}\oplus\cO_{\p1}(-1),\\
\cE(2\xi)\otimes\cO_{N_j}\cong\cO_{\p1}(1)\oplus\cO_{\p1},
\end{gather*}
hence $h^2\big(F,\cE\otimes\cI_{Z\vert F}(2\xi)\big)=0$.
\qed
\medbreak

Recall that $\cI_F(\alpha\xi^2+\beta  \xi f)$ has been defined in the introduction as the locus of points representing instanton bundles with charge $\alpha\xi^2+\beta  \xi f$ in the moduli space $\cM_F(2;0,\alpha\xi^2+\beta  \xi f)$ of $\mu$--stable vector bundles with respect to $\cO_F(h)$. The following corollary is almost immediate.

\begin{corollary}
\label{cInstanton}
For each $\alpha,\beta\in\bZ$  such that $\alpha\ge2$, $\alpha+\beta\ge4$ and $4\alpha+3\beta\ge15$ there is an irreducible component 
$$
\cI_F^{0}(\alpha\xi^2+\beta  \xi f)\subseteq\cI_F(\alpha\xi^2+\beta  \xi f)
$$
 which is  generically smooth of dimension $8\alpha+6\beta-30$ and containing all the points corresponding to the bundles obtained via Construction \ref{conInstanton}.
\end{corollary}
\begin{proof}
The schemes as in Equality \eqref{ZInstanton} represent points in a non--empty open subset $\cU\subseteq\Lambda^{\times\alpha-2}\times \Lambda_N^{\times\alpha+\beta-4}$. Since the latter product is irreducible (see Remarks \ref{rLine} and \ref{rPseudoLine}), it follows that $\cU$ is irreducible as well. 

Since the  bundle $\cE$ in Sequence \eqref{seqStandard} is uniquely determined by the scheme $Z$, we obtain in this way a flat family of bundles  containing all the bundles obtained via Construction \ref{conInstanton} and parameterized by $\cU$. Thus we deduce the existence of a morphism $u\colon \cU\to \cI_F(\alpha\xi^2+\beta  \xi f)$. Every point in $u(\cU)$ is smooth because $\Ext^2_{F}\big(\cE,\cE\big)=0$ (see Theorem \ref{tInstanton}), thus there is a unique component $\cI_F^{0}(\alpha\xi^2+\beta  \xi f)$ containing $u(\cU)$: Theorem \ref{tInstanton} then implies
$$
\dim\cI_F^{0}(\alpha\xi^2+\beta  \xi f)=\dim \Ext^1_{F}\big(\cE,\cE\big)=8\alpha+6\beta-30.
$$
This last equality completes the proof of the corollary.
\end{proof}

\begin{remark}
\label{rNonEarnest}
The bundles constructed in the previous proof are certainly not earnest if $\alpha\ge4$, thanks to Corollary \ref{cSimplify}. 

Indeed, the cohomology of Sequence \eqref{seqStandard} tensored by $\cO_F(-E)\cong\cO_F(-\xi+f)$ and Proposition \ref{pLineBundle} yield the exact sequence
$$
0\longrightarrow H^1\big(F,\cE(-E)\big)\longrightarrow H^1\big(F,\cI_{Z\vert F}(-2\xi)\big)\longrightarrow \bC,
$$
hence $h^1\big(F,\cE(-E)\big)\ge h^1\big(F,\cI_{Z\vert F}(-2\xi)\big)-1$.
In order to compute $h^1\big(F,\cI_{Z\vert F}(-2\xi)\big)$ we consider the cohomology of Sequence \eqref{seqIdeal} tensored by $\cO_F(-2\xi)$, taking into account that $h^0\big(F,\cO_F(-2\xi)\big)=h^1\big(F,\cO_F(-2\xi)\big)=0$ and 
\begin{gather*}
\cO_F(-2\xi)\otimes\cO_{L_i}\cong\cO_{\p1},\\
\cO_F(-2\xi)\otimes\cO_{N_j}\cong\cO_{\p1}(-2).
\end{gather*}
It follows that $h^1\big(F,\cI_{Z\vert F}(-2\xi)\big)=h^0\big(Z,\cO_F(-2\xi)\otimes\cO_Z\big)=\alpha-2$, hence 
$$
\alpha-2\ge h^1\big(F,\cE(-E)\big)\ge\alpha-3.
$$
\end{remark}

\section{Existence of earnest instanton bundles on the blow up of $\p3$}
\label{sEarnest}
In this section we complete the study of instanton bundles on the blow up $F:=F_1$ of $\p3$ along a line: again $\xi$ denotes $\xi_1$. In spite of the previous Remark \ref{rNonEarnest}, a different choice of the scheme $Z$ allows us to construct earnest instanton bundles $\cE$ on $F$ which are generically trivial and $\mu$--stable with $c_2(\cE)=\alpha\xi^2+\beta \xi f$ for each admissible non--negative integers $\alpha,\beta$.

\begin{construction}
\label{conEarnest}
Let $\alpha$ and $\beta$ be integers such that $\alpha\ge2$, $\beta\ge1$ and $4\alpha+3\beta\ge15$. We take $M_1,\dots,M_{\alpha-2}$ and $N_1,\dots, N_{\beta-1}$ pairwise disjoint curves corresponding to points in $\Lambda_M$ and $\Lambda_N$ respectively and define 
\begin{equation}
\label{ZEarnest}
Z:=\bigcup_{i=1}^{\alpha-2}M_i\cup\bigcup_{j=1}^{\beta-1}N_j\subseteq F.
\end{equation}
Notice that the restriction $4\alpha+3\beta\ge15$ implies $Z\ne\emptyset$.

We claim that $\det(\cN_{Z\vert F})\cong\cO_F(\xi+f)\otimes\cO_Z$. We check such an isomorphism component by component: indeed
\begin{gather*}
\det(\cN_{Z\vert F})\otimes\cO_{M_i}\cong\cO_{\p1}(2)\cong\cO_F(\xi+f)\otimes\cO_{M_i},\\
\det(\cN_{Z\vert F})\otimes\cO_{N_j}\cong\cO_{\p1}(1)\cong\cO_F(\xi+f)\otimes\cO_{N_j},
\end{gather*}
thanks to Remarks \ref{rPseudoConic} and \ref{rPseudoLine}.

The equality $h^2\big(F,\cO_F(-\xi-f)\big)=0$ and Theorem \ref{tSerre} guarantee the existence of a vector bundle $\cF$ on $F$ with a section $s$ vanishing exactly along $Z$ and with $c_1(\cF)=\xi+f$, $c_2(\cF)=Z$, fitting into Sequence \eqref{seqSerre}. Tensoring such sequence by $\cO_F(-2\xi-f)$ we obtain the exact sequence
\begin{equation}
\label{seqEarnest}
0\longrightarrow\cO_F(-2\xi-f)\longrightarrow\cE\longrightarrow\cI_{Z\vert F}(-\xi)\longrightarrow0,
\end{equation}
where $\cE:=\cF(-2\xi-f)$.

The bundle $\cE$ is uniquely determined by $Z$, because $h^1\big(F,\cO_F(-\xi-f)\big)=0$.
\end{construction}

The main result of the section is the following proof of Theorem \ref{tEarnest} stated in the introduction.

\medbreak
\noindent{\it Proof of Theorem \ref{tEarnest}.}
We trivially have $c_1(\cE)=-h$ and $c_2(\cE)=\alpha\xi^2+\beta\xi f$ by construction.  Arguing as in the proof of Theorem \ref{tInstanton} one easily obtains from the cohomology of Sequence \eqref{seqEarnest} that $h^1\big(F,\cE\big)=0$. Let us prove that $\cE$ is $\mu$--stable, i.e. that $h^0\big(F,\cE(-a\xi-bf)\big)=0$ for  each pair of integers $a$ and $b$ satisfying Inequality \eqref{AB}. We will check this by showing that 
$$
h^0\big(F,\cO_F(-(a+2)\xi-(b+1)f)\big)=h^0\big(F,\cI_{Z\vert F}(-(a+1)\xi-bf)\big)=0
$$
in that range, again computing the cohomology of Sequence \eqref{seqEarnest}.

This is obvious if either $a\ge0$. If $a\le-1$, then 
$$
-(a+2)-(b+1)\le \frac23 a\le -1.
$$
Similarly
\begin{equation}
\label{Hard}
-(a+1)-b\le 2+\frac23 a\le -1,
\end{equation}
for $a\le-4$.  

Argueing as above, if $a=-1$, then the only cases we need to handle are $b=0,-1$, because all the other values of $b$ satisfying Inequality \eqref{AB} satisfy Inequality \eqref{Hard} as well. If $b=0$, then we have to check the vanishing $h^0\big(F,\cI_{Z\vert F}\big)=0$, which is trivial because $Z\ne\emptyset$. If $b=-1$, then we have to check the vanishing $h^0\big(F,\cI_{Z\vert F}(f)\big)=0$. 
If $\alpha\ge3$, then no fibres of $\pi$ contain a curve in $\Lambda_M$. 
If $\alpha=2$, then the restriction on the charge forces $\beta\ge 3$ and the vanishing is still trivial because no fibre of $\pi$ can contain two or more disjoint curves in $\Lambda_N$. 

If $a=-2$, then we have only to deal with $b=1$, i.e. we have to check that $h^0\big(F,\cI_{Z\vert F}(E)\big)=0$ which is easy to check: similarly for the case $a=-3$. It follows that $\cE$ is an instanton bundle. 

Restricting Sequence \eqref{seqEarnest} to a general line $L\in\Lambda$ one deduces that $\cE$ is generically trivial. In order to show that $\cE$ is earnest we can use the same argument of Remark \ref{rNonEarnest}. The cohomology of Sequence \eqref{seqEarnest} tensored by $\cO_F(-\xi+f)$ and Proposition \ref{pLineBundle} yield $h^1\big(F,\cE(-\xi+f)\big)=h^1\big(F,\cI_{Z\vert F}(-2\xi+f)\big)$. The cohomology of Sequence \eqref{seqIdeal} tensored by $\cO_F(-2\xi+f)$ yields 
$$
h^1\big(F,\cI_{Z\vert F}(-2\xi+f)\big)=h^0\big(Z,\cO_F(-2\xi+f)\otimes\cO_Z\big).
$$
Finally, since
\begin{gather*}
\cO_F(-2\xi+f)\otimes\cO_{M_i}\cong\cO_{\p1}(-1),\\
\cO_F(-2\xi+f)\otimes\cO_{N_j}\cong\cO_{\p1}(-2),
\end{gather*}
it follows that $h^0\big(Z,\cO_F(-2\xi+f)\otimes\cO_Z\big)=0$. Thus $\cE$ is earnest, thanks to Corollary \ref{cSimplify}.

As in the proof of Theorem \ref{tInstanton} we know that $\cE$ is simple, and $\Ext^3_{F}\big(\cE,\cE\big)=0$. It remains to check that $h^2\big(F,\cE\otimes\cE^\vee\big)=0$ computing the cohomology of  Sequence \eqref{seqStandard} tensored by $\cE^\vee\cong\cE(h)$: again it suffices to check that 
$$
h^2\big(F,\cE(\xi)\big)=h^2\big(F,\cE\otimes\cI_{Z\vert F}(2\xi+f)\big)=0.
$$
Thanks to Equality \eqref{Serre1}, the former vanishing has been proved in Proposition \ref{pTable}, because $h^2\big(F,\cE(\xi)\big)=h^1\big(F,\cE(-\xi)\big)=0$. The latter can be obtained imitating verbatim the argument for proving the analogous vanishing in the proof of Theorem \ref{tInstanton}.
\qed
\medbreak

In particular we have proved the existence of earnest instanton bundles inside $\cI_F(\alpha\xi^2+\beta\xi f)$. The same argument of the proof of Corollary \ref{cInstanton} also proves the following corollary.

\begin{corollary}
\label{cEarnest}
For each $\alpha,\beta\in\bZ$  such that $\alpha\ge2$, $\beta\ge1$ and $4\alpha+3\beta\ge15$ there is an irreducible component 
$$
\cI_F^1(\alpha\xi^2+\beta  \xi f)\subseteq\cI_F(\alpha\xi^2+\beta  \xi f)
$$
 which is  generically smooth of dimension $8\alpha+6\beta-30$ and containing all the points corresponding to the bundles obtained via Construction \ref{conEarnest}.
\end{corollary}
\begin{proof}
The schemes as in Equality \eqref{ZEarnest} represent points in a non--empty open subset $\cV\subseteq\Lambda_M^{\times\alpha-2}\times \Lambda_N^{\times\beta-1}$ which is irreducible. Thus the proof runs along the same lines of the proof of Corollary \ref{cInstanton}.
\end{proof}

\section{Some remarks and questions on the blow up of $\p3$}
\label{sFinal}
In this section we collect some comments and questions on the structure of the moduli space of instanton bundles on $F:=F_1$.

In view of the irreducibility of the moduli space of instanton bundles on $\p3$ recently proved in \cite{Tik1,Tik2}, the following question seems to be natural.

\begin{question}
Is it true that the scheme
$$
\cI_F(\alpha\xi^2+\beta  \xi f)
$$
is irreducible and smooth?
\end{question}

Let us deal with the above question for instanton bundles $\cE$ of  {\sl minimal charge}, i.e. instanton bundles whose charge has minimal degree. Corollary \ref{cMinimal} implies that $c_2(\cE)h\ge15$. When equality holds, argueing as in the proof of Corollary \ref{cMinimal}, we deduce $c_2(\cE)=3\xi^2+\xi f$ and we have the following affirmative answer to the above question
$$
\cI_F(3\xi^2+  \xi f)=\{\ \Omega_{F\vert\p1}(-f)\ \},
$$
thanks to the proposition below.

\begin{proposition}
\label{pMinimal}
If $\cE$ is an instanton bundle on $F$ with $c_2(\cE)h=15$, then $\cE\cong\Omega_{F\vert\p1}(-f)$.
\end{proposition}
\begin{proof}
As pointed out above we know that if there is an instanton bundle $\cE$ with $c_2(\cE)h=15$, then $c_2(\cE)=3\xi^2+\xi f$. Moreover Construction \ref{conInstanton} with $\alpha=3$ and $\beta=1$ guarantees the existence of at least one such instanton bundle $\cE$.

We now prove that $\cE\cong\Omega_{F\vert\p1}(-f)$. Since $\alpha=3$, $\beta=1$, it follows from  by Corollary \eqref{cBound} that $\delta=2\gamma$. On the one hand, the cohomology of Sequence \eqref{seqEuler} tensored by $\cE$ and Proposition \ref{pTable} return
$$
h^3\big(F,\cE(-\xi-f)\big)=h^2\big(F,\cE(-\xi+f)\big)=h^1\big(F,\cE(-\xi+f)\big)=\gamma.
$$
On the other hand, Equality \eqref{Serre1} and Lemma \ref{lHoppe} yield 
$$
h^3\big(F,\cE(-\xi-f)\big)=h^0\big(F,\cE(\xi+f)\big)=0.
$$
It follows  $\delta=0$ and that $\cE$ is earnest, thanks to Corollary \ref{cSimplify}. Thus Theorem \ref{tSimplify} implies that $\cE$ is the cohomology of Monad \eqref{Monad} with $\alpha=3$, $\beta=1$ and $\gamma=\delta=0$, hence $\cE\cong\Omega_{F\vert\p1}(-f)$.
\end{proof}

\begin{remark}
By combining the above Proposition and Theorem \ref{tEarnest} we also deduce that $\Omega_{F\vert\p1}(-f)$ can be obtained via Construction \ref{conEarnest} starting from a general section in $H^0\big(F,\Omega_{F\vert\p1}(2\xi)\big)$.
\end{remark}

Recall that $\alpha+\beta\ge4$ for each instanton bundle $\cE$ with $c_2(\cE)=\alpha\xi^2+\beta\xi f$. Thus the aforementioned bundle $\Omega_{F\vert\p1}(-f)$ can be viewed as a particular case of instanton bundles such that $\beta=4-\alpha$, i.e. with charge $\alpha\xi^2+\beta\xi f$. We spend some words about such bundles in what follows.

In the case $\alpha+\beta=4$ one has $\delta=2\gamma$. As in the proof of Theorem \ref{tSimplify} the cohomology of Sequences \eqref{Display} tensored by $\cO_F(2\xi)$ returns $h^0\big(F,\cE(2\xi)\big)\ge2$.

Let $s\in H^0\big(F,\cE(2\xi)\big)$ be a non--zero section. Then $(s)_0=C\cup S$ where $C$ is either empty, or a subscheme of pure codimension $2$ and $S$ is either $0$, or $S\in \vert a\xi+bf\vert$ with $a\ge0$, $a+b\ge0$. We deduce that $\cE(2\xi-S)$ has a section vanishing on $C$, hence $h^0\big(F,\cE((2-a)\xi-bf)\big)\ne0$. Since $\cE$ is $\mu$--semistable, it follows from Lemma \ref{lHoppe} that $15(a-2)+9b\le-27$, hence $2a \le 5a+3b\le 1$. 

Thus $S=0$ necessarily and the general $s\in  H^0\big(F,\cE(2\xi)\big)$ vanishes exactly along a subscheme $C\subseteq F$ of pure codimension $2$ whose class in $A^2(F)$ is $c_2(\cE(2\xi))=(\alpha-2)(\xi^2-\xi f)$. Thus $\cE$ fits into an exact sequence of the form
$$
0\longrightarrow\cO_F(-2\xi)\longrightarrow\cE\longrightarrow\cI_{C\vert F}(-\xi-f)\longrightarrow0.
$$

Since $CE=2-\alpha$ and $\alpha\ge3$ because $c_2(\cE)h\ge15$ (see Corollary \ref{cMinimal}), we deduce that a component $Y$ of $C$ is contained in $E$. The line bundle $\cO_F(\xi)$ is globally generated, hence both $Y\xi$ and $(C-Y)\xi$ must be non--negative. We deduce that the class of $Y$ inside $A^2(F)$ is $u(\xi^2-\xi f)$ for some positive $u\in\bZ$. The natural injection $\Pic(E)\cong A^1(E)\subseteq A^2(F)$ yields that the support of $Y$ is actually a line in $\Lambda\cong\p1$, because $E\cong\p1\times\p1$. 

Let $\alpha\ge4$. On the one hand $\cE$ cannot be earnest because $0\ge\beta\ge 1-\gamma$, thanks to Corollary \eqref{cBound}. On the other hand, in Construction \ref{conInstanton} we defined a rational map
$$
\lambda\colon Sym^{\alpha-2}\Lambda\dashrightarrow\cI_F(\alpha\xi^2+\beta \xi f)
$$
on the complement of the union of the diagonals. Its image is contained in the component $\cI_F^0(\alpha\xi^2+\beta\xi f)$. The map $\lambda$ can never be dominant because the fibre at a point in $\im(\lambda)$ has dimension $h^0\big(F,\cE(2\xi)\big)-1\ge1$.

In particular, unreduced schemes supported on lines in $\Lambda$ play a non--trivial role in the structure of $\cI_F(\alpha\xi^2+\beta \xi f)$. 

The discussion above shows that the study of the irreducibility and smoothness of $\cI_F(\alpha\xi^2+  \beta\xi f)$ could be quite hard in general. Nevertheless, when $\alpha,\beta\in\bZ$ satisfy $\alpha\ge2$, $\beta\ge1$ and $4\alpha+3\beta\ge15$, we constructed in the previous section at least the two irreducible components $\cI_F^0(\alpha\xi^2+\beta  \xi f)$ and $\cI_F^1(\alpha\xi^2+\beta  \xi f)$. 

Let $\cI_F^{earnest}(\alpha\xi^2+\beta\xi f)$ be the closure inside $\cI_F(\alpha\xi^2+\beta\xi f)$ of the locus of points representing earnest bundles. The condition $h^1\big(F,\cE(-E)\big)$ is open on flat family, hence $\cI_F^1(\alpha\xi^2+\beta\xi f)\subseteq\cI_F^{earnest}(\alpha\xi^2+\beta\xi f)$. Moreover, if $\alpha=2$, then $\cI_F^0(\alpha\xi^2+\beta\xi f)\subseteq\cI_F^{earnest}(\alpha\xi^2+\beta\xi f)$, thanks to Remark \ref{rNonEarnest}.

Thus the following perhaps simpler question arises naturally.

\begin{question}
\label{qEquality}
Is it true that
$$
\cI_F^0(\alpha\xi^2+\beta  \xi f)=\cI_F^1(\alpha\xi^2+\beta  \xi f)=\cI_F^{earnest}(\alpha\xi^2+ \beta \xi f)
$$
when $\beta$ is a positive integer?
\end{question}

We already described above the trivial case $\alpha=3$ and $\beta=1$ corresponding to $c_2(\cE)h=15$. When $c_2(\cE)h=16$ one easily checks $\beta\le0$. Thus the first non--trivial case is $c_2(\cE)h=17$, which results in $\alpha=2$ and $\beta=3$ when $\beta$ is positive. We will show below that the answer to the above question is affirmative in this case. To this purpose it suffices to check that $\cI_F^{earnest}(2\xi^2+ 3 \xi f)$ is irreducible of dimension $4$.

Let $\cE$ be an earnest instanton bundle with $c_2(\cE)=2\xi^2+3\xi f$. The cohomology of Sequences \eqref{seqEuler} tensored by $\cE(\xi)$ and $\cE(\xi+f)$ yields $h^0\big(F,\cE(\xi+2f)\big)=1$.

Let $s\in H^0\big(F,\cE(\xi+2f)\big)$ be a non--zero section. Then $(s)_0=C\cup S$ where $C$ is either empty, or a subscheme of pure codimension $2$ and $S$ is either $0$, or $S\in \vert a\xi+bf\vert$ with $a\ge0$, $a+b\ge0$. We deduce that $\cE(-S)$ has a section vanishing on $C$, hence $h^0\big(F,\cE((1-a)\xi+(2-b)f)\big)\ne0$. Since $\cE$ is $\mu$--semistable, it follows from Lemma \ref{lHoppe} that $15(a-1)+9(b-2)\le-27$, hence $2a \le 5a+3b\le 2$. Thus either $a=1$, hence $b=-1$, or $S=0$.

The case $a=1$ and $b=-1$ does not occur. Indeed, on the one hand, we checked above that $h^0\big(F,\cE(3f)\big)\ne0$. On the other hand the cohomology of Sequences \eqref{seqEuler} tensored by $\cE(f)$ and $\cE(2f)$ returns $h^0\big(F,\cE(3f)\big)=0$, a contradiction. 

We deduce that $S=0$. Since $c_2(\cE(\xi+2f))=0$, it follows that $\cE$ fits into a sequence of the form
\begin{equation*}
%\label{%seqExtension}
0\longrightarrow\cO_F(-\xi-2f)\longrightarrow\cE\longrightarrow\cO_F(-2\xi+f)\longrightarrow0.
\end{equation*}
Since
$$
\dim\Ext^1_F\big(\cO_F(-2\xi+f),\cO_F(-\xi-2f)\big)=h^1\big(F,\cO_F(\xi-3f)\big)=5,
$$
it follows that $\cI_F^{earnest}(2\xi^2+ 3 \xi f)$ is isomorphic to a non--empty open subset of $\p4$. In particular $\cI_F^{earnest}(2\xi^2+ 3 \xi f)$ is irreducible of dimension $4$, which is what we claimed above.

\begin{remark}
One can easily prove using Lemma \ref{lHoppe}  that each non--zero element $\Ext^1_F\big(\cO_F(-2\xi+f),\cO_F(-\xi-2f)\big)$  returns a $\mu$--semistable instanton bundle.

The above discussion implies that the general element actually induces an earnest, generically trivial, $\mu$--stable instanton bundle.
\end{remark}

\begin{remark}
It is not difficult to check that the unique value of the charge such that there are instanton bundles which are extensions of line bundles is exactly $2\xi^2+3\xi f$.
\end{remark}

\section{Instanton bundles on $\p1\times\p2$}
\label{sSegre}
In this last section we will describe the due changes to the arguments used in the previous sections for dealing with instanton bundles on $F_0=\p1\times\p2$. Again we will omit the subscript $e=0$ in the formulas, thus we will simply write  $F$, $\xi$, $\sigma$, $\cP$, $h$, $\cC^\bullet$ for $F_0$, $\xi_0$, $\sigma_0$, $\cP_0$, $h_0$, $\cC^\bullet_0$.  

In this case $\xi^3=0$, $\xi^2f=1$ and $3(1+e)a+9(a+b)=12a+9b$ in Lemma \ref{lHoppe}. Moreover, Sequence \eqref{seqOmega}  is the pull--back of the standard Euler sequence on $\p2$ via $\sigma$ and
$$
\Omega_{F\vert \p1}\cong\sigma^*\Omega_{\p2},\qquad \Omega_{F} \cong\cO_{F}(-2f)\oplus\sigma^*\Omega_{\p2}.
$$

The first step is to compute the cohomology of $\cO_F(a\xi+b f)$.

\begin{proposition}
\label{pLineBundleSegre}
We have
$$
\begin{aligned}
h^0\big(F,\cO_F(a\xi+b f)\big)&={a+2\choose2}{b+1\choose1},\\
h^1\big(F,\cO_F(a\xi+b f)\big)&={a+2\choose2}{-1-b\choose1},\\
h^2\big(F,\cO_F(a\xi+b f)\big)&={-1-a\choose2}{b+1\choose1},\\
h^3\big(F,\cO_F(a\xi+b f)\big)&={-1-a\choose2}{-1-b\choose1}.
\end{aligned}
$$
\end{proposition}
\begin{proof}
It suffices to apply the K\"unneth formulas.
\end{proof}

We follow the same methods used in the previous sections. Indeed we set 
\begin{align*}
(\cF_{-5},\cF_{-4}&,\cF_{-3},\cF_{-2},\cF_{-1},\cF_0):=\\
&:=(\cO_F(-\xi-f),\cO_F(-\xi),\cO_F(-f),\cO_F,\cO_F(\xi-f),\cO_F(\xi)),
\end{align*}
\begin{align*}
(\cG_{0},\cG_{1}&,\cG_{2},\cG_{3},\cG_{4},\cG_5):=\\
&:=(\cO_F(-\xi),\cO_F(-\xi-f),\Omega_{F\vert\p1},\Omega_{F\vert\p1}(-f),\cO_F(-2\xi),\cO_F(-2\xi-f)),
\end{align*}
(the Orlov collection with respect to $\cO_F(\xi)$ and its dual tensored by $\cO_F(\xi)$ and $\cO_F(-\xi)$ respectively: see \cite[Corollary 2.6]{Orl}).
\begin{lemma}
\label{lAOSegre}
Let $\cE$ be an instanton bundle on $F$.

Then $\cE$  is the cohomology in degree $0$ of a complex $\widehat{\cC}^\bullet$ with
$$
{\cC}^i:= \bigoplus_{q+p=i}E^{p,q}:=H^{q+\lceil\frac p2\rceil}\big(F,\cE\otimes\cF_{-p}\big)\otimes\cG_{p}.
$$
\end{lemma}
\begin{proof}
The proof is the same as the one of Lemma \ref{lAO}.
\end{proof}

In order to prove Theorem \ref{tSimplifySegre} we compute below $e^{p,q}:=H^{q+\lceil\frac p2\rceil}\big(F,\cE\otimes\cF_{-p}\big)$ for $0\le q\le 5$ and $-5\le p\le 0$.

\begin{proposition}
\label{pTableSegre}
Let $\cE$ be an instanton bundle on $F$.

If $c_2(\cE)=\alpha\xi^2+\beta \xi f$ then $e^{p,q}$ is the number in position $(p,q)$ in the following table.
\begin{table}[H]
\centering
\bgroup
\def\arraystretch{1.5}
\begin{tabular}{ccccccc}
\cline{1-6}
\multicolumn{1}{|c|}{0} & \multicolumn{1}{c|}{0} & \multicolumn{1}{c|}{0} & \multicolumn{1}{c|}{0} & \multicolumn{1}{c|}{0} & \multicolumn{1}{c|}{0} & $q=5$ \\ \cline{1-6}
\multicolumn{1}{|c|}{$\alpha+\beta-6$} & \multicolumn{1}{c|}{$\beta-3$} & \multicolumn{1}{c|}{0} & \multicolumn{1}{c|}{0} & \multicolumn{1}{c|}{0} & \multicolumn{1}{c|}{0} & $q=4$ \\ \cline{1-6}
\multicolumn{1}{|c|}{0} & \multicolumn{1}{c|}{0} & \multicolumn{1}{c|}{$\alpha-2$} & \multicolumn{1}{c|}{0} & \multicolumn{1}{c|}{0} & \multicolumn{1}{c|}{0} & $q=3$ \\ \cline{1-6}
\multicolumn{1}{|c|}{0} & \multicolumn{1}{c|}{0} & \multicolumn{1}{c|}{0} & \multicolumn{1}{c|}{0} & \multicolumn{1}{c|}{$\alpha-\beta+\gamma$} & \multicolumn{1}{c|}{0} & $q=2$ \\ \cline{1-6}
\multicolumn{1}{|c|}{0} & \multicolumn{1}{c|}{0} & \multicolumn{1}{c|}{0} & \multicolumn{1}{c|}{0} & \multicolumn{1}{c|}{$\gamma$} & \multicolumn{1}{c|}{$\beta-3$}& $q=1$ \\ \cline{1-6}
\multicolumn{1}{|c|}{0} & \multicolumn{1}{c|}{0} & \multicolumn{1}{c|}{0} & \multicolumn{1}{c|}{0} & \multicolumn{1}{c|}{0} & \multicolumn{1}{c|}{0} & $q=0$ \\ \cline{1-6}
$p=-5$ & $p=-4$ & $p=-3$ & $p=-2$ & $p=-1$ & $p=0$
\end{tabular}
\egroup
\caption{The values of $e^{p,q}$}
\end{table}
\end{proposition}
\begin{proof}
For each $(p,q)$ but $(-4,3)$, $(0,2)$ and $(-1,1)$ the corresponding values of $e^{p,q}$ are obtained repeating word by word the arguments in the proof of Proposition \ref{pTable}.

By definition $e^{-1,1}=h^1\big(F,\cE\otimes\cF_{-1}\big)=\gamma$. Moreover, Equality \eqref{Serre1} implies $e^{-4,3}=e^{0,2}$. The cohomology of Sequence \eqref{seqOmega} and its dual tensored by $\cE(\xi)$ and $\cE(-2\xi)$ respectively and the vanishings $e^{-4,5}=e^{-2,3}=0$ yield
\begin{align*}
e^{0,2}=h^2\big(F,\cE\otimes\cF_{0}\big)=h^2\big(F,\cE(\xi)\big)= h^3\big(F,\cE\otimes\Omega_{F\vert \p1}(\xi)\big)=0.
\end{align*}  

The statement is then completely proved.
\end{proof}

The following corollary and proof of Theorem \ref{tSimplifySegre} are immediate.

\begin{corollary}
\label{cBoundSegre}
Let $\cE$ be an instanton bundle with $c_2(\cE)=\alpha\xi^2+\beta  \xi f$ on $F$.

Then $\alpha\ge2$, $\beta\ge3$ and $\alpha+\beta\ge6$.
\end{corollary}
\begin{proof}
The statement follows from the obvious non--negativity of the $e_{p,q}$'s.
\end{proof}

\medbreak
\noindent{\it Proof of Theorem \ref{tSimplifySegre}.}
The proof is completely analogous to the one of Theorem \ref{tSimplify}. 

If $\cE$ is an instanton bundle on $F$, then it suffices to apply \ref{lAOSegre} using the values $e^{p,q}$ calculated in Proposition \ref{pTableSegre} in order to obtain the complex ${\cC}^\bullet$ where
\begin{gather*}
\cC^{-1}:=\cO_F(-2\xi-f)^{\oplus\alpha+\beta-6},\\
\cC^0:=\cO_F(-2\xi)^{\oplus\beta-3}\oplus\Omega_{F\vert \p1}(-f)^{\oplus\alpha-2}\oplus\cO_F(-\xi-f)^{\oplus\gamma},\\
\cC^1:=\cO_F(-\xi-f)^{\oplus\alpha-\beta+\gamma}\oplus\cO_F(-\xi)^{\oplus\beta-3}.
\end{gather*}

Conversely, let $\cE$ be the cohomology $\cC^\bullet$. Argueing as in the analogous part of the proof of Theorem \ref{tSimplify} one deduces that $c_1(\cE)=-3\xi-2f$, $c_2(\cE)=\alpha\xi^2+\beta  \xi f$, $h^1\big(F,\cE(\xi-f)\big)\le\gamma$ and that $\cE$ is earnest.
\qed
\medbreak

\begin{remark}
\label{rReferee2}
Again the $\mu$--semistability of the cohomology $\cE$ of the monad $\cC^\bullet$ in the second part of the statement of Theorem \ref{tSimplifySegre} is necessary.

Indeed, the same argument used in Remark \ref{rReferee2} leads to a surjective morphism $\varphi\colon \cO_F(-2\xi)\oplus\cO_F(-\xi-f)^{\oplus2}\to\cO_F(-\xi)$. Thus we still obtain Monad \eqref{MonadSegre} when $\alpha=2$, $\beta=4$ and $\gamma=2$, whose cohomology $\cE\cong\ker(\varphi)\cong\cO_F(-2\xi)\oplus\cO_F(-\xi-2f)$, which is not $\mu$--semistable.
\end{remark}

We now prove the existence of instanton bundles via the Serre construction.

\begin{construction}
\label{conSegre}
Let $\alpha$ and $\beta$ be integers such that $\alpha\ge2$, $\beta\ge3$ and $\alpha+\beta\ge6$. We take $M_1,\dots, M_{\alpha-2}$ and $N_1,\dots,N_{\beta-3}$ pairwise disjoint curves corresponding to points in $\Lambda_M$ and $\Lambda_N$ respectively and define 
\begin{equation}
\label{ZSegre}
Z:=\bigcup_{i=1}^{\alpha-2}M_i\cup\bigcup_{j=1}^{\beta-3}N_j\subseteq F.
\end{equation}
Notice that the restriction $\alpha+\beta\ge6$ implies $Z\ne\emptyset$.

Since we have the isomorphisms (see Remarks \ref{rPseudoConic}, \ref{rPseudoLine} and Equality \eqref{Normal})
\begin{gather*}
\det(\cN_{Z\vert F})\otimes\cO_{M_i}\cong\cO_{\p1}\cong\cO_F(\xi)\otimes\cO_{M_i},\\
\det(\cN_{Z\vert F})\otimes\cO_{N_j}\cong\cO_{\p1}(1)\cong\cO_F(\xi)\otimes\cO_{N_j},
\end{gather*}
it follows that $\det(\cN_{Z\vert F})\cong\cO_F(\xi)\otimes\cO_Z$. 

Thus the equality $h^2\big(F,\cO_F(-\xi)\big)=0$ and Theorem \ref{tSerre} guarantee the existence of a vector bundle $\cF$ on $F$ with a section $s$ vanishing exactly along $Z$ and with $c_1(\cF)=\xi$, $c_2(\cF)=Z$, fitting into Sequence \eqref{seqSerre}. Tensoring such sequence by $\cO_F(-2\xi-f)$ we obtain the exact sequence
\begin{equation}
\label{seqSegre}
0\longrightarrow\cO_F(-2\xi-f)\longrightarrow\cE\longrightarrow\cI_{Z\vert F}(-\xi-f)\longrightarrow0,
\end{equation}
where $\cE:=\cF(-2\xi-f)$.

The bundle $\cE$ is uniquely determined by $Z$, because $h^1\big(F,\cO_F(-\xi)\big)=0$.
\end{construction}

We are now able to prove Theorem \ref{tSegre} stated in the introduction.

\medbreak
\noindent{\it Proof of Theorem \ref{tSegre}.}
By construction $c_1(\cE)=-h$ and $c_2(\cE)=\alpha\xi^2+\beta\xi f$. One easily obtains $h^1\big(F,\cE\big)=0$ from the cohomology of Sequence \eqref{seqSegre}. Let us prove that $\cE$ is $\mu$--stable, i.e. that $h^0\big(F,\cE(-a\xi-bf)\big)=0$ for  each pair of integers $a$ and $b$ such that $12a+9b\ge\mu(\cE)=-27$, i.e. satisfying Inequality 
\begin{equation}
\label{ABSegre}
b\ge-3-\frac43a.
\end{equation}
We will check this by showing that 
\begin{equation}
\label{h^0}
h^0\big(F,\cO_F(-(a+2)\xi-(b+1)f)\big)=h^0\big(F,\cI_{Z\vert F}(-(a+1)\xi-(b+1)f)\big)=0
\end{equation}
in that range. This is obvious if either $a\ge0$, or $b\ge0$. 

Let $a,b\le-1$: if $a\le-2$, then Inequality \eqref{ABSegre} implies $b\ge0$ and the assertion follows from the former case.

Let $a=-1$: Inequality \eqref{ABSegre} implies $-(b+1)\le0$, i.e. $b\ge-1$ hence again the statement follows from Proposition \ref{pLineBundleSegre} unless $b=-1$. In this case Equalities \eqref{h^0} are trivial. It follows that $\cE$ is an instanton bundle. 

Thanks to Remark \ref{rLine} we know that $F$ does not contain lines, hence $\cE$ is generically trivial by definition. Since $\cE$ is an instanton bundle on $F$, it follows that it is the cohomology of Monad \eqref{MonadSegre}, hence it is automatically earnest thanks to Theorem \ref{tSimplifySegre}.

We know that $\cE$, being $\mu$--stable, is also simple, hence $\Ext^3_{F}\big(\cE,\cE\big)=0$. The vanishing $h^2\big(F,\cE\otimes\cE^\vee\big)=0$ follows from the cohomology of  Sequence \eqref{seqSegre} tensored by $\cE^\vee\cong\cE(h)$, once we check that $h^2\big(F,\cE(\xi+f)\big)=h^2\big(F,\cE\otimes\cI_{Z\vert F}(2\xi+f)\big)=0$.

Thanks to Proposition \ref{pLineBundleSegre} we know that $h^2\big(F,\cO_F(-\xi)\big)=0$, hence the cohomology of Sequences \eqref{seqSegre} tensored by $\cO_F(\xi+f)$ and \eqref{seqIdeal} return
$$
h^2\big(F,\cE(\xi+f)\big)\le h^2\big(F,\cI_{Z\vert F}\big)\le h^1\big(Z,\cO_Z\big).
$$
The dimension on the right is zero, because $Z$ is the disjoint union of smooth rational curves.

A similar argument shows that $h^2\big(F,\cE(2\xi+f)\big)\le h^1\big(Z,\cO_Z(\xi)\big)=0$, hence the cohomology of Sequence \eqref{seqIdeal} tensored by $\cE(2\xi+f)$ returns
\begin{align*}
h^2\big(F&,\cE\otimes\cI_{Z\vert F}(2\xi+f)\big)\le h^1\big(Z,\cO_Z\otimes\cE(2\xi+f)\big)=\\
&=\sum_{i=1}^{\alpha-1}h^1\big(N_i,\cO_{N_i}\otimes\cE(2\xi+f)\big)+\sum_{j=1}^{\beta-3}h^1\big(M_j,\cO_{M_j}\otimes\cE(2\xi+f)\big).
\end{align*}
Equality \eqref{Normal} yields $\cE(2\xi+f)\otimes\cO_Z\cong\cN_{Z\vert F}$, hence 
$$
\cO_{N_i}\otimes\cE(2\xi+f)\cong\cO_{N_i}\oplus\cO_{N_i}(1),\qquad \cO_{M_j}\otimes\cE(2\xi+f)\cong\cO_{M_j}^{\oplus2}.
$$
hence $h^2\big(F,\cE\otimes\cI_{Z\vert F}(2\xi+f)\big)=0$.
\qed
\medbreak

In particular we have proved the existence of earnest instanton bundles inside $\cI_F(\alpha\xi^2+\beta\xi f)$. 

\begin{corollary}
\label{cInstantonSegre}
For each $\alpha,\beta\in\bZ$  such that $\alpha\ge2$, $\beta\ge3$ and $\alpha+\beta\ge6$ there is an irreducible component inside $\cI_F(\alpha\xi^2+\beta  \xi f)$ which is  generically smooth of dimension $4\alpha+6\beta-30$ and containing all the points corresponding to the bundles obtained via Construction \ref{conSegre}.
\end{corollary}
\begin{proof}
The schemes as in Equality \eqref{ZSegre} represent points in a non--empty open subset $\cW\subseteq\Lambda_M^{\times\alpha-2}\times \Lambda_N^{\times\beta-3}$ which is irreducible  (see Remarks \ref{rPseudoConic} and \ref{rPseudoLine}). Thus we deduce the statement as in the proofs of Corollaries \ref{cInstanton} and \ref{cEarnest}.
\end{proof}

Let $\cE$ be an instanton bundle with $c_2(\cE)=\alpha\xi^2+\beta\xi f$. Thus, Inequality \eqref{Minimal} yields $c_2(\cE)h=2\alpha+3\beta\ge14$. Moreover, it is easy to check using Corollary \ref{cBoundSegre} that the case $c_2(\cE)h=14$ cannot occur. In particular, an instanton bundle $\cE$ of minimal charge still satisfies $c_2(\cE)h=15$. The description of such an $\cE$ is easy, thanks to Theorem \ref{tSimplifySegre}.

\begin{proposition}
\label{pMinimalSegre}
If $\cE$ is an instanton bundle on $F$ with $c_2(\cE)h=15$, then $\cE\cong\Omega_{F\vert\p1}(-f)$.
\end{proposition}
\begin{proof}
The restrictions $\alpha+\beta\ge6$ and $\beta\ge3$ imply $c_2(\cE)h=2\alpha+3\beta\ge15$ for each instanton bundle on $F$. 

If equality holds, then $\alpha+\beta=6$, hence $\alpha=\beta=3$ and we know that such an $\cE$ exists thanks to Construction \ref{conSegre}. The same argument used in the proof of Proposition \ref{pMinimal} still shows that $\gamma=0$, hence still yields $\cE\cong\Omega_{F\vert\p1}(-f)$ thanks to  Theorem \ref{tSimplifySegre}.
\end{proof}

\bigskip
\noindent
Gianfranco Casnati,\\
Dipartimento di Scienze Matematiche, Politecnico di Torino,\\
c.so Duca degli Abruzzi 24,\\
10129 Torino, Italy\\
e-mail: {\tt gianfranco.casnati@polito.it}

\bigskip
\noindent
Ozhan Genc,\\
Dipartimento di Scienze Matematiche, Politecnico di Torino,\\
c.so Duca degli Abruzzi 24,\\
10129 Torino, Italy\\
\textit{Current Address: Department of Mathematics and Informatics,\\
Jagiellonian University, {\L}ojasiewicza 6, 30-348 Krak{\'o}w, Poland}\\
e-mail: {\tt ozhangenc@gmail.com}


\begin{thebibliography}{44}

\bibitem{A--O}
V. Ancona, G. Ottaviani: {\em Canonical resolutions of sheaves on Schubert and Brieskorn varieties}. In \lq Complex analysis\rq\ (Wuppertal, 1991), K. Diederich ed., Aspects Math., E17, Friedr. Vieweg, Braunschweig, (1991), 14--19.

\bibitem{A--M}
V. Antonelli, F. Malaspina: {\em Instanton bundles on the Segre threefold with Picard number three}. arXiv:1909.10895 [math.AG], to appear in Math. Nachr..

\bibitem{Ar}
  E. Arrondo: \emph{A home--made Hartshorne--Serre correspondence}. Comm. Algebra  \textbf{ 20} \rm (2007),  423--443.

\bibitem{A--D--H--M}
M.F. Atiyah, N.J. Hitchin, V.G. Drinfel’d, Y.I. Manin: \emph{Construction of instantons}. Phys. Lett. A, \textbf{65} (1978), 185--187.


 
%\bibitem{C--G--H--S}M. Casanellas, R. Hartshorne, F. Geiss, F.O. Schreyer: \emph{Stable Ulrich bundles}. Int. J. of Math. \textbf{23} 1250083 \rm(2012). 

\bibitem{C--C--G--M}
G. Casnati, E. Coskun, O. Genc, F. Malaspina: \emph{Instanton bundles on the blow up of $\p3$ at a point}. arXiv:1909.10281 [math.AG],  to appear in Mich. Math. J..
 
%\bibitem{E--S--W} 
%D. Eisenbud, F.-O. Schreyer, J. Weyman: \emph{Resultants and Chow forms via exterior syzygies}. J. Amer. Math. Soc. \textbf{16} \rm (2003),  537--579.

\bibitem{Fa}
D. Faenzi: \emph{Even and odd instanton bundles on Fano threefolds of Picard number one}. Manuscripta Math. \textbf{144} (2014), 199--239.
 
\bibitem{Ha2}
R. Hartshorne: {\em Algebraic geometry}. G.T.M. 52, Springer \rm (1977).

\bibitem{Ha3}
R. Hartshorne: {\em Coherent functors}. Adv. Math. \textbf{140} (1998), 44--94.

\bibitem{Hop}
H. Hoppe: {\em Generischer Spaltungstyp und zweite Chernklasse stabiler Vektorraumb\"undel vom Rang $4$ auf $\p4$}. Math. Z. \textbf{187} (1984), 345--360.

\bibitem{I--P}
  V.A. Iskovskikh, Yu.G. Prokhorov: \emph{Fano varieties}. Algebraic Geometry V (A.N. Parshin and I.R. Shafarevich eds.), Encyclopedia of Mathematical Sciences  47, Springer, \rm (1999).

\bibitem{J--M--P--S}
M.B. Jardim, G. Menet, D.M. Prata, H.N. S\'a Earp: {\em Holomorphic bundles for higher dimensional gauge theory}. Bull. London Math. Soc.
\textbf{49} (2017), 117--132.

\bibitem{Kuz}
A. Kuznetsov: {\em Instanton bundles on Fano threefolds}. Cent. Eur. J. Math. \textbf{10} (2012), 1198--1231.

\bibitem{Laz2}
R. Lazarsfeld: \emph{Positivity in algebraic geometry. II. Positivity for Vector Bundles, and Multiplier Ideals}. A Series of Modern Surveys in Mathematics, 49. Springer \rm (2004).

\bibitem{M--M--PL}
F. Malaspina, S. Marchesi, J. Pons--Llopis: {\em Instanton bundles on the flag variety $F(0,1,2)$}. arXiv:1706.06353 [math.AG],  to appear in Ann. Sc. Norm. Pisa. 

\bibitem{Ma}
M. Maruyama: {\em Boundedness of semistable sheaves of small ranks}. Nagoya Math. J. \textbf{78} (1980), 65--94.
 
\bibitem{Orl}
D.O. Orlov: {\em Projective bundles, monoidal transformations, and derived categories of coherent sheaves}. (Russian) Izv. Ross. Akad. Nauk Ser. Mat. \textbf{56} (1992), 852--862; translation in Russian Acad. Sci. Izv. Math. \textbf{41} (1993), 133--141.

\bibitem{Tik1}
A. S. Tikhomirov: {\em Moduli of mathematical instanton vector bundles with odd $c_2$ on projective space}. (Russian) Izv. Ross. Akad. Nauk Ser. Mat. \textbf{76} (2012), 143--224; translation in Izv. Math. \textbf{76} (2012), 991--1073.

\bibitem{Tik2}
A.S. Tikhomirov: {\em Moduli of mathematical instanton vector bundles with even $c_2$ on projective space}. (Russian) Izv. Ross. Akad. Nauk Ser. Mat. \textbf{77} (2013), 139--168; translation in Izv. Math. \textbf{77} (2013), 1195--1223.


\end{thebibliography}
\end{document}